\documentclass{amsart}
\newtheorem{theorem}{Theorem}[section]
\newtheorem{lemma}[theorem]{Lemma}
\newtheorem{proposition}[theorem]{Proposition}
\newtheorem{corollary}[theorem]{Corollary}

\newtheorem{criterion}[theorem]{Criterion}
\theoremstyle{definition}

\newtheorem{example}[theorem]{Example}

\usepackage{amscd,amssymb,systeme,bm}
\begin{document}

\title[Varieties of bicommutative algebras]{Varieties of bicommutative algebras\\ with identity of degree three}

\author[Vesselin Drensky, Bekzat Zhakhayev]{Vesselin Drensky, Bekzat Zhakhayev}
\thanks{The second author was supported by the Science Committee of the Ministry of Science and Higher Education of the Republic of Kazakhstan  (Grant No. AP26199089)}
\address{Institute of Mathematics and Informatics, Bulgarian Academy of Sciences, Sofia 1113, Bulgaria}
\email{drensky@math.bas.bg}
\address{SDU University, Abylaikhan Str. 1/1,  040900, Kaskelen, Kazakhstan.\newline
Institute of Mathematics and Mathematical Modeling, Shevchenko Str. 28, 050010, Almaty, Kazakhstan}
\email{bekzat.kopzhasar@gmail.com}

\subjclass[2020]{17A30, 17A50, 20C30.}
\keywords{varieties of bicommutative algebras, lattices of subvarieties, cocharacters.}

\begin{abstract}
The variety of bicommutative algebras is the class of all nonassociative algebras satisfying the polynomial identities
$(x_1x_2)x_3=(x_1x_3)x_2$ and $x_1(x_2x_3)=x_2(x_1x_3)$. In this paper we provide a complete description of varieties of bicommutative algebras
over a field of characteristic zero that satisfy a polynomial identity of degree three.
Furthermore, we establish a sufficient and necessary condition for a variety of bicommutative algebras to have a distributive lattice of subvarieties.
\end{abstract}

\maketitle

\section{Introduction}
The variety of bicommutative algebras $\mathfrak B$ over a field $K$ is a variety of nonassociative algebras defined by the polynomial identities
of right- and left-commutativity
\[
(x_1x_2)x_3=(x_1x_3)x_2,\quad x_1(x_2x_3)=x_2(x_1x_3).
\]
In the sequel we consider algebras over a field $K$ of characteristic 0 only.
The origins of one-sided commutative algebras can be traced back to the paper by Cayley \cite{Ca} in 1857,
see the paper by Dzhumadil'daev, Ismailov, and Tulenbaev \cite{DIT} and our paper \cite{DrZ} for historical details.
The description of the free bicommutative algebra $F({\mathfrak B})$ of countable rank over a field $K$ of arbitrary characteristic
was obtained by Dzhumadil'daev, Ismailov, and Tulenbaev \cite{DIT} and by Dzhumadil'daev and Tulenbaev \cite{DT}.
There the authors constructed an explicit basis of $F({\mathfrak B})$ as a vector space over $K$.
This allowed them to compute the Hilbert (or Poincar\'e) series of the free algebra $F_d({\mathfrak B})$ of finite rank $d$,
the cocharacter (in characteristic 0) and the codimension sequences $\chi_n({\mathfrak B})$ and $c_n({\mathfrak B})$, $n=1,2,\ldots$,
and to show that the exponent $\lim_{n\to\infty}(c_n({\mathfrak B}))^{1/n}$ exists and equals 2.
In \cite{DrZ} we showed that bicommutative algebras have many properties similar to those of commutative algebras
and established that varieties of bicommutative algebras over a field of a characteristic $p\geq 0$ have finite bases of their polynomial identities.
The paper \cite{Dr3} by one of the authors of the present paper studied in detail the subvarieties of $\mathfrak B$ and their numerical invariants in characteristic 0.
In particular, the description of the varieties generated by a two-dimensional bicommutative algebra was given
and it was shown that $\mathfrak B$ is generated by one of these two-dimensional algebras.

The main purpose of our paper is to give a complete description of the varieties of bicommutative algebras satisfying a polynomial identity of degree 3.
Under the natural action of the symmetric group $S_n$ the multilinear component $P_n({\mathfrak B})$ of degree $n$ is a direct sum of the irreducible $S_n$-modules $M(\lambda)$,
where $\lambda$ is a partition of $n$, participating with multiplicity $m(\lambda)$ given in \cite{DIT}:
\[
P_n({\mathfrak B})=\bigoplus_{\lambda\vdash n}m(\lambda)M(\lambda),\,\lambda=(\lambda_1,\lambda_2),
\]
\begin{equation}\label{multiplicities of cocharacters of B}
m(\lambda)=\begin{cases}
1,\lambda=(1);\\
n-1,\lambda=(n),n>1;\\
n-2\lambda_2+1, \lambda=(\lambda_1,\lambda_2), \lambda_2>0.
\end{cases}
\end{equation}
In the early 1980s Berele \cite{B} and Drensky \cite{Dr1} showed that for the needs of PI-algebras representation theory of the general linear group $GL_d(K)$
can be applied and the obtained results can be easily translated in the language of representation theory of the symmetric group.
From this point of view the equation (\ref{multiplicities of cocharacters of B}) for $n=3$
\begin{equation}\label{decomposition degree 3}
P_3({\mathfrak B})=2M(3)\oplus 2M(2,1)
\end{equation}
implies that the irreducible $S_3$-components of $P_3({\mathfrak B})$ can be generated by the linearizations of
\begin{equation}\label{PI (3)}
f_{(3)}(x)=\alpha_1 x(xx)+\alpha_2 (xx)x,\,(\alpha_1,\alpha_2)\not=(0,0),
\end{equation}
\begin{equation}\label{PI (2,1)}
f_{(2,1)}(x,y)=\beta_1x(xy-yx)+\beta_2(xy-yx)x,\,(\beta_1,\beta_2)\not=(0,0).
\end{equation}

We shall describe the varieties which satisfy one of the identities (\ref{PI (3)}) and (\ref{PI (2,1)}).
The results will be given in the language of graph theory. Our approach is similar to the approach of Vladimirova and Drensky \cite{VDr}
where the authors described the varieties of associative algebras satisfying a polynomial identity of third degree.

The set of all subvarieties of a variety $\mathfrak W$ of algebraic systems forms a lattice $L({\mathfrak W})$ with respect to the operations
\[
{\mathfrak W}_1\vee{\mathfrak W}_2=\text{var}({\mathfrak W}_1\cup{\mathfrak W}_2),\;{\mathfrak W}_1\wedge{\mathfrak W}_2={\mathfrak W}_1\cap{\mathfrak W}_2,\;
{\mathfrak W}_1,{\mathfrak W}_2\subseteq {\mathfrak W}.
\]
The lattice $L({\mathfrak W})$ is distributive if
\[
{\mathfrak W}_1\wedge({\mathfrak W}_2\vee{\mathfrak W}_3)=({\mathfrak W}_1\wedge{\mathfrak W}_2)\vee({\mathfrak W}_1\wedge{\mathfrak W}_3),\;
{\mathfrak W}_1,{\mathfrak W}_2,{\mathfrak W}_3\subseteq{\mathfrak W}.
\]
The distributivity of the lattice of subvarieties is equivalent to the condition that in the cocharacter sequence
\[
\chi_n({\mathfrak W})=\sum_{\lambda\vdash n}m(\lambda)\chi_{\lambda},\,n=1,2,\ldots,
\]
all  multiplicities satisfy the condition $m(\lambda)\leq 1$.
Ananin and Kemer \cite{AK} described the varieties of associative algebras with a distributive lattice of subvarieties.
They showed that this happens if and only if  $m(2,1)\leq 1$ in the third cocharacter $\chi_3({\mathfrak W})$ of $\mathfrak W$.
We prove that the similar restriction $m(3),m(2,1)\leq 1$ is a sufficient and necessary condition for the distributivity of the lattice of the subvarieties
of ${\mathfrak W}\subset\mathfrak B$. In the language of polynomial identities this means that $\mathfrak W$ satisfies both identities
(\ref{PI (3)}) and (\ref{PI (2,1)}) for some $(\alpha_1,\alpha_2),(\beta_1,\beta_2)\not=(0,0)$.

\section{Preliminaries}
\subsection{Representations of the symmetric and the general linear groups}
For a background on representation theory of the symmetric group $S_n$ and general linear group $GL_d(K)$ we refer e.g. to the books by James and Kerber \cite{JK} and Weyl \cite{W}.
The irreducible representations of the symmetric group $S_n$ are described by partitions $\lambda=(\lambda_1,\ldots,\lambda_n)$ of $n$ (notation $\lambda\vdash n$).
Similarly, the irreducible polynomial representations of the general linear group $GL_d(K)$ are described by partitions in not more than $d$ parts.
We shall denote by $M(\lambda)$ and $W_d(\lambda)$ the corresponding to $\lambda$ modules of $S_n$ and $GL_d(K)$, respectively.
The group $S_n$ acts from the left on the vector space $P_n$ of the multilinear polynomials of degree $n$ in the free associative algebra $K\langle X\rangle$ by
\[
\sigma(f(x_1,\ldots,x_n))=f(x_{\sigma(1)},\ldots, x_{\sigma(n)}),\;\sigma\in S_n,f(x_1,\ldots,x_n)\in P_n.
\]
The group $GL_d(K)$ acts canonically on the vector space $KX_d$ with basis $X_d=\{x_1,\ldots,x_d\}$
and this action is extended diagonally on the homogeneous component $K\langle X_d\rangle^{(n)}$ of degree $n$ of $K\langle X_d\rangle$ by
\[
g(v(x_1,\ldots,x_d))=v(g(x_1),\ldots,g(x_d)),\;g\in GL_d(K),v(x_1,\ldots,x_d)\in K\langle X_d\rangle^{(n)}.
\]
The symmetric group $S_n$ acts from the right on $K\langle X_d\rangle^{(n)}$ by
\[
(x_{i_1}\cdots x_{i_n})^{\tau^{-1}}=x_{i_{\tau(1)}}\cdots x_{i_{\tau(n)}},\;x_{i_1}\cdots x_{i_n}\in K\langle X_d\rangle^{(n)}, \tau\in S_n.
\]
The irreducible $GL_d(K)$-submodule $W_d(\lambda)=W_d(\lambda_1,\ldots,\lambda_d)$ of $K\langle X_d\rangle^{(n)}$ contains an element $w_{\lambda}(x_1,\ldots,x_d)$
which is called the highest weight vector of $W_d(\lambda)$ and is multihomogeneous of degree $(\lambda_1,\ldots,\lambda_d)$ and can be obtained in the following way.
Let $d_1,\ldots,d_{\lambda_1}$ be the lengths of the columns of the Young diagram of $\lambda$. Then
\[
w_{\lambda}(x_1,\ldots,x_d)=\sum_{\tau\in S_n}\alpha_{\tau}\big(\prod_{k=1}^{\lambda_1}S_{d_k}(x_1,\ldots,x_{d_k})\big)^{\tau},\;\alpha_{\tau}\in K,
\]
where
\[
S_d(x_1,\ldots,x_d)=\sum_{\rho\in S_d}(-1)^{\rho}x_{\rho(1)}\cdots x_{\rho(d)}
\]
is the standard polynomial (and $(-1)^{\rho}$ is the sign of $\rho$).
The linearization of $w_{\lambda}(x_1,\ldots,x_d)$ generates an irreducible $S_n$-submodule $M(\lambda)$ of $P_n$ and every $M(\lambda)$ with $\lambda_{d+1}=0$
can be obtained in such a way.

There is a simple criterion for a multihomogeneous polynomial  $w(x_1,\ldots,x_d)$ of degree $(\lambda_1,\ldots,\lambda_d)$
to be a highest weight vector. Recall that a derivation of an algebra $R$ is a linear map $\Delta:R\to R$ such that
\[
\Delta(uv)=\Delta(u)v+u\Delta(v),\, u,v\in R.
\]
Define a derivation $\Delta_{ij}$, $i,j=1,\ldots,d$, $i\not=j$, of $K\langle X_d\rangle$ by
\[
\Delta_{ij}(x_i)=x_j,\,\Delta_{ij}(x_k)=0,\,k\not=i.
\]
More generally, we define the derivation $\Delta_i(u,\ast)$, $u\in K\langle X_d\rangle$, by
\[
\Delta_i(u,x_i)=u,\,\Delta_i(u,x_k)=0,\,k\not=i.
\]

\begin{criterion}\label{criterion hwv}
Let $w(x_1,\ldots,x_d)\in K\langle X_d\rangle$ be a nonzero multihomogeneous polynomial of degree $(\lambda_1,\ldots,\lambda_d)$.
Then $w(x_1,\ldots,x_d)\not=0$ is a highest weight vector of $W_d(\lambda)\subset K\langle X_d\rangle$ if and only if
$\Delta_{ij}(w)=0$ for all $1\leq j<i\leq d$.
\end{criterion}

The following proposition is similar to \cite[Lemma 2.5]{Dr2} and \cite[Proposition 1]{DrK}.

\begin{proposition}\label{consequences of degree n+1}
Let $M$ be a submodule of $P_n$. Then all consequences of degree $n+1$ of the polynomial identities from $M$ can be obtained in the following way:

{\rm (i)} The $S_{n+1}$-submodule of $P_{n+1}$ generated by the products $Mx_{n+1}$ is a homomorphic image of the $S_{n+1}$-module $(M\otimes M(1))\uparrow S_{n+1}$
induced by the tensor product of the $S_n$-module $M$ and $S_1$-module $M(1)$ and similarly from the consequences generated by the products $x_{n+1}M$.
Here $S_1$ acts on the one-element set $\{n+1\}$ and $M\otimes M(1)$ is considered as a module of $S_n\times S_1\subset S_{n+1}$.

{\rm (ii)} The $S_{n+1}$-submodule of $P_{n+1}$ generated by the substitutions
\[
f(x_1,\ldots,x_{n-1},x_nx_{n+1}), f(x_1,\ldots,x_{n-1},x_n)\in M
\]
is a homomorphic image of $((M\downarrow S_{n-1})\otimes P_2)\uparrow S_{n+1}$. Here $M\downarrow S_{n-1}$ is considered as an $S_{n-1}$-module, $S_2$ acts on $\{n,n+1\}$ and
$(M\downarrow S_{n-1})\otimes P_2$ is an $S_{n-1}\times S_2$-module, $S_{n-1}\times S_2\subset S_{n+1}$.
\end{proposition}

The proof repeats verbatim the proofs in \cite{Dr2} and \cite{DrK}.
Similar considerations hold for any variety of algebras over a field of characteristic 0.

\begin{example}\label{consequences in partitions}
Below we give the decomposition of the $S_{n+1}$-modules in Proposition \ref{consequences of degree n+1}.
The calculations are easy exercises on the Branching theorem and the Littlewood-Richardson rule.

(i) $(M\otimes M(1))\uparrow S_{n+1}$:

\noindent $\lambda=(\lambda_1,\lambda_2)\vdash n$, $\lambda_1>\lambda_2\geq 0$:
\[
(M(\lambda)\otimes M(1))\uparrow S_{n+1}=M(\lambda_1+1,\lambda_2)\oplus M(\lambda_1,\lambda_2+1);
\]
$\lambda=(\lambda_1,\lambda_1)\vdash n\geq 2$:
\[
(M(\lambda)\otimes M(1))\uparrow S_{n+1}=M(\lambda_1+1,\lambda_1).
\]

(ii) $((M\downarrow S_{n-1})\otimes M(2))\uparrow S_{n+1}$:

\noindent $\lambda=(n)$, $n\geq 3$:
\[
(M(n)\downarrow S_{n-1})\otimes M(2))\uparrow S_{n+1}=M(n+1)\oplus M(n,1)\oplus M(n-1,2);
\]
$\lambda=(\lambda_1,\lambda_2)$, $\lambda_1\geq \lambda_2+3$, $\lambda_2\geq 2$:
\[
M(\lambda_1,\lambda_2)\downarrow S_{n-1}=M(\lambda_1,\lambda_2-1)+M(\lambda_1-1,\lambda_2),
\]
\[
(M(\lambda_1,\lambda_2-1)\otimes M(2))\uparrow S_{n+1}=M(\lambda_1+2,\lambda_2-1)+M(\lambda_1+1,\lambda_2)
\]
\[
+M(\lambda_1+1,\lambda_2-1,1)+M(\lambda_1,\lambda_2+1)+M(\lambda_1,\lambda_2,1)+M(\lambda_1,\lambda_2-1,2);
\]
\[
(M(\lambda_1-1,\lambda_2)\otimes M(2))\uparrow S_{n+1}=M(\lambda_1+1,\lambda_2)+M(\lambda_1,\lambda_2+1)+M(\lambda_1,\lambda_2,1)
\]
\[
+M(\lambda_1-1,\lambda_2+2)+M(\lambda_1-1,\lambda_2+1,1)+M(\lambda_1-1,\lambda_2+2)
\]

(iii) $((M\downarrow S_{n-1})\otimes M(1^2))\uparrow S_{n+1}$:

\noindent $\lambda=(n)$, $n\geq 2$:
\[
(M(n)\downarrow S_{n-1})\otimes M(1^2))\uparrow S_{n+1}=M(n,1)+M(n-1,1^2);
\]
$\lambda=(\lambda_1,\lambda_2)$, $\lambda_1\geq \lambda_2+2$, $\lambda_2\geq 2$:
\[
M(\lambda_1,\lambda_2)\downarrow S_{n-1}=M(\lambda_1-1,\lambda_2)+M(\lambda_1,\lambda_2-1),
\]
\[
(M(\lambda_1-1,\lambda_2)\otimes M(1^2))\uparrow S_{n+1}=M(\lambda_1,\lambda_2+1)+M(\lambda_1,\lambda_2,1)
\]
\[
+M(\lambda_1-1,\lambda_2+1)+M(\lambda_1-1,\lambda_2,1^2),
\]
\[
(M(\lambda_1,\lambda_2-1)\otimes M(1^2))\uparrow S_{n+1}=M(\lambda_1+1,\lambda_2)+M(\lambda_1+1,\lambda_2-1,1)
\]
\[
+M(\lambda_1,\lambda_2,1)+M(\lambda_1,\lambda_2-1,1^2).
\]
\end{example}

\begin{example}\label{consequences of M(3) in associative case}
The following are the consequences of degree 4 in two variables for
\[
u_{(3)}(x_1)=x_1^3\in W_d(3)\subset K\langle X_d\rangle:
\]
(i) $u^{(1)}_{(4)}=u_{(3)}(x_1)x_1$, $u^{(1)}_{(3,1)}(x_1,x_2)=\Delta_{12}(u_{(3)}(x_1))x_1-3u_{(3)}(x_1)x_2$
and similarly we obtain $u^{(2)}_{(4)}=x_1u_{(3)}(x_1)$ and $u^{(2)}_{(3,1)}=x_1\Delta_{12}(u_{(3)}(x_1))-3x_2u_{(3)}(x_1)$
by multiplication from the left.

(ii) $u^{(3)}_{(4)}=\Delta_1(x_1^2,u_{(3)}(x_1))$. The complete linearization of $u_{(3)}(x_1)$ is
\[
h_{(3)}(x_1,x_2,x_3)=\sum_{\sigma\in S_3}x_{\sigma(1)}x_{\sigma(2)}x_{\sigma(3)},
\]
\[
u^{(3)}_{(3,1)}=\Delta_1(x_1x_2+x_2x_1,u_{(3)}(x_1))-h_{(3)}(x_1,x_2,x_1^2)=2x_1x_2x_1^2-x_1^3x_2-x_1^2x_2x_1,
\]
\[
u^{(1)}_{(2^2)}=\Delta_1(x_2^2,u_{(3)}(x_1))+\Delta_2(x_1^2,u_{(3)}(x_2))-\frac{1}{2}h_{(3)}(x_1,x_2,x_1x_2+x_2x_1).
\]

(iii) $u^{(4)}_{(3,1)}=\Delta_1([x_1,x_2],u_{(3)}(x_1))$, where $[x_1,x_2]=x_1x_2-x_2x_1$ is the commutator of $x_1$ and $x_2$.
\end{example}

\begin{example}\label{consequences of M(2,1) in associative case}
The following are the consequences of degree 4 in two variables for
\[
u_{(2,1)}(x_1,x_2)=\beta_1x_1[x_1,x_2]+\beta_2[x_1,x_2]x_1\in W_d(2,1)\subset K\langle X_d\rangle:
\]
(i) $u^{(1)}_{(3,1)}=u_{(2,1)}(x_1,x_2)x_1$, $u^{(1)}_{(2^2)}=u_{(2,1)}(x_1,x_2)x_2+u_{(2,1)}(x_2,x_1)x_1$
and similarly $u^{(2)}_{(3,1)}=x_1u_{(2,1)}(x_1,x_2)$ and $u^{(2)}_{(2^2)}=x_1u_{(2,1)}(x_1,x_2)+x_1u_{(2,1)}(x_2,x_1)$ by multiplication from the left.

(ii) $u^{(1)}_{(4)}=\Delta_2(x_1^2,u_{(2,1)}(x_1,x_2))$. Let
\[
h_{(2,1)}(x_1,x_2,x_3)=\Delta_{12}(u_{(2,1)}(x_1,x_3))
\]
\[
=\beta_1(x_1[x_2,x_3]+x_2[x_1,x_3])+\beta_2([x_1,x_3]x_2+[x_2,x_3]x_1)
\]
be the linearization of $u_{(2,1)}(x_1,x_2)$. Then
\[
h_{(2)}(x_1,x_2,x_3)=h_{(2,1)}(x_1,x_2,x_3),
\]
\[
h_{(1^2)}(x_1,x_2,x_3)=h_{(2,1)}(x_1,x_2,x_3)+2h_{(2,1)}(x_1,x_3,x_2)
\]
are, respectively, the generators of the $S_2$-submodules $M(2)$ and $M(1^2)$ of the $S_2$-module $M(2,1)\downarrow S_2$.
Applying Criterion \ref{criterion hwv} we obtain that the following polynomials are highest weight vectors:
\[
u^{(3)}_{(3,1)}=u_{(2,1)}(x_1,x_1x_2+x_2x_1),
\]
\[
u^{(3)}_{(2^2)}=h_{(2)}(x_1,x_2,x_1x_2+x_2x_1)=h_{(2,1)}(x_1,x_2,x_1x_2+x_2x_1),
\]
\[
u^{(4)}_{(3,1)}=h_{(1^2)}(x_1,x_2,x_1^2]).
\]

(iii) $u^{(5)}_{(3,1)}=h_{(2)}(x_1,x_1,[x_1,x_2])$, $u^{(4)}_{(2^2)}=h_{(1^2)}(x_1,x_2,[x_1,x_2])$.
\end{example}

\subsection{The free bicommutative algebra}
The free bicommutative algebra $F({\mathfrak B})$ was described in \cite{DT} and \cite{DIT}.
In was shown there that the square $A^2$ of any bicommutative algebra $A$ is a commutative and associative algebra.
The idea for the construction of the basis of $F({\mathfrak B})$ given in \cite{DIT} was used in our paper \cite{DrZ},
where we found the following algebra $G_d$ which is isomorphic to $F_d({\mathfrak B})$.
Let
\[
K[Y_d,Z_d]=K[y_1,\ldots,y_d,z_1,\ldots,z_d]
\]
be the polynomial algebra in $2d$ variables
and let $\omega(K[Y_d])$ and $\omega(K[Z_d])$ be the augmentation ideals of $K[Y_d]$ and $K[Z_d]$, respectively.
Then the algebra $G_d$ has a basis
\[
X_d\cup \{Y_d^{\alpha}Z_d^{\beta}\mid Y_d^{\alpha}Z_d^{\beta}\in \omega(K[Y_d])\omega(K[Z_d])\},
\]
where $Y_d^{\alpha}Z_d^{\beta}=y_1^{\alpha_1}\cdots y_d^{\alpha_d}z_1^{\beta_1}\cdots z_d^{\beta_d}$
and the multiplication in $G_d$ is defined by
\[
x_ix_j=y_iz_j,\,x_i(Y_d^{\alpha}Z_d^{\beta})=y_iY_d^{\alpha}Z_d^{\beta},\,(Y_d^{\alpha}Z_d^{\beta})x_i=Y_d^{\alpha}Z_d^{\beta}z_i,
\]
\[
(Y_d^{\alpha}Z_d^{\beta})(Y_d^{\gamma}Z_d^{\delta})=Y_d^{\alpha+\gamma}Z_d^{\beta+\delta}.
\]
Among the consequences of a homogeneous polynomial $f(X_d)\in F_d({\mathfrak B})$ of degree $\geq 2$
there are such which are obtained by substitutions of a variable by linear combination of variables
 and by products of variables. If the image of $f(X_d)\in F_d({\mathfrak B})$ in $G_d$ is $u(Y_d,Z_d)$
 then the images of $f(x_1,\ldots,x_{i-1},x_j+x_k,x_{i+1},\ldots,x_d)$ and $f(x_1,\ldots,x_{i-1},x_jx_k,x_{i+1},\ldots,x_d)$
 in $G_d$ are, respectively,
 \[
 u(y_1,\ldots,y_{i-1},y_j+y_k,y_{i+1},\ldots,y_d,z_1,\ldots,z_{i-1},z_j+z_k,z_{i+1},\ldots,z_d),
 \]
 \[
 u(y_1,\ldots,y_{i-1},y_jz_k,y_{i+1},\ldots,y_d,z_1,\ldots,z_{i-1},y_jz_k,z_{i+1},\ldots,z_d).
 \]
We shall use the same notation $\Delta_{ij}$ for the derivation of the algebra $G_d$
which is an analogue of the derivation $\Delta_{ij}$, $i,j=1,\ldots,d$, $i\not=j$, of $K\langle X_d\rangle$. It acts on $X_d,Y_d,Z_d$ by
\[
\Delta_{ij}(x_i)=x_j,\,\Delta_{ij}(y_i)=y_j,\,\Delta_{ij}(z_i)=z_j,\,\Delta_{ij}(x_k)=\Delta_{ij}(y_k)=\Delta_{ij}(z_k)=0,\,k\not=i.
\]
Similarly, if $u=u(Y_d,Z_d)\in G_d^2$, then
\[
\Delta_i(u):x_i,y_i,z_i\to u(Y_d,Z_d),\,\Delta_i(u):x_k,y_k,z_k\to 0,\,k\not=i.
\]
In the sequel we shall identify the algebras $F_d({\mathfrak B})$ and $G_d$ and shall speak for example about the polynomial identity
$f(Y_d,Z_d)=0$ instead of its preimage in $F_d({\mathfrak B})$.

The highest weight vectors of the $GL_d(K)$-submodules $W_d(\lambda)$ of $G_d$ were described in \cite[Lemma 3.3]{Dr3}:

\begin{lemma}\label{hwv of G}
The following polynomials $w_{\lambda}^{(k)}$ form a maximal linearly independent system of highest weight vectors of the $GL_d(K)$-submodules $W_d(\lambda)$ in $G_d^2$:
\begin{equation}\label{hwv of W(lambda)}
\begin{split}
w_{(n)}^{(j)}=y_1^jz_1^{n-j},&\quad j=1,2,\ldots,n-1,\\
w_{\lambda}^{(j)}=y_1^j(y_1z_2-y_2z_1)^{\lambda_2}z_1^{\lambda_1-\lambda_2-j}, &\quad j=0,1,\ldots,\lambda_1-\lambda_2,\text{ if }\lambda_2>0.
\end{split}
\end{equation}
\end{lemma}

In particular, the polynomials in (\ref{PI (3)}) and (\ref{PI (2,1)}) are, respectively,
special cases for $\lambda=(3)$ and $\lambda=(2,1)$ of the highest weight vectors in Lemma \ref{hwv of G}.
Their images in $G_d$ are
\begin{equation}\label{PI (3) in G}
w_{(3)}(y_1,z_1)=(\alpha_1y_1+\alpha_2z_1)(y_1z_1)=\alpha_1w_{(3)}^{(2)}+\alpha_2w_{(3)}^{(1)},
\end{equation}
\begin{equation}\label{PI (2,1) in G}
w_{(2,1)}(y_1,y_2,z_1,z_2)=(\beta_1y_1+\beta_2z_1)(y_1z_2-y_2z_1)=\beta_1w_{(2,1)}^{(1)}+\beta_2w_{(2,1)}^{(0)},
\end{equation}
$(\alpha_1,\alpha_2)\not=(0,0)$, $(\beta_1,\beta_2)\not=(0,0)$.

The following statement from \cite[Lemma 4.2]{Dr3} will allow to restrict our considerations
in the search for consequences of the identities (\ref{PI (3) in G}) and (\ref{PI (2,1) in G}).

\begin{lemma}\label{consequences of w(lambda)}
If $0\not=f\in W(\lambda_1,\lambda_2)\subset F({\mathfrak B})$, then all polynomial identities $w_{(\mu_1,\mu_2)}^{(j)}=0$ with $\mu_2\geq \lambda_1$
are consequences of the polynomial identity $f=0$.
\end{lemma}

In the sequel we denote, respectively, by ${\mathfrak U}_{\alpha}$ and ${\mathfrak V}_{\beta}$ the varieties of bicommutative algebras
defined by the polynomial identities (\ref{PI (3)}) and (\ref{PI (2,1)}).
Since for the defining identities identities we have, respectively, $\lambda_1=3$ and $\lambda_1=2$.
Since $\chi_n({\mathfrak B})$ is a sum of $S_n$-characters $\chi_{\lambda}$ which $\lambda=(\lambda_1,\lambda_2)$ only,
Lemma \ref{consequences of w(lambda)} gives immediately:

\begin{corollary}\label{bounds for cocharacters}
\[
\chi_n({\mathfrak U}_{\alpha})=m_{\alpha}(n)\chi_{(n)}+m_{\alpha}(n-1,1)\chi_{(n-1,1)}+m_{\alpha}(n-2,2)\chi_{(n-2,2)},
\]
\[
\chi_n({\mathfrak V}_{\alpha})=m_{\beta}(n)\chi_{(n)}+m_{\beta}(n-1,1)\chi_{(n-1,1)}.
\]
\end{corollary}

\section{The variety defined by the identity $\alpha_1 x(xx)+\alpha_2 (xx)x=0$}
Translating in the language of $G_d$ the identity (\ref{PI (3)}) becomes by (\ref{PI (3) in G})
\[
w_{(3)}(y_1,z_1)=(\alpha_1y_1+\alpha_2z_1)(y_1z_1)=\alpha_1w_{(3)}^{(2)}+\alpha_2w_{(3)}^{(1)}=0.
\]
We shall denote by $G_d({\mathfrak U}_{\alpha})$ the factor algebra of $G_d$ modulo the T-ideal generated by this identity.
Clearly, $G_d({\mathfrak U}_{\alpha})$ is isomorphic to the relatively free algebra $F_d({\mathfrak U}_{\alpha})$.

Transferring the results in Example \ref{consequences of M(3) in associative case} for the consequences of the identity $u_{(3)}(x_1)=x_1^3$
in the following lemma we obtain all consequences of degree 4 of the identity $w_{(3)}(y_1,z_1)=0$.
The highest weight vectors are expressed as linear combinations of the highest weight vectors in (\ref{hwv of W(lambda)}).
Their explicit form is found with the help of Criterion \ref{criterion hwv}.

\begin{lemma}\label{consequences of M(3) in G}
All consequences of degree $4$ in $G_d({\mathfrak U}_{\alpha})$ of the identity {\rm (\ref{PI (3)})} are equivalent to the identities obtained in the following way:

{\rm (i)} By one side multiplication:
\[
v^{(1)}_{(4)}=w_{(3)}(y_1,z_1)z_1=(\alpha_1y_1+\alpha_2z_1)(y_1z_1)z_1=\alpha_1w_{(4)}^{(2)}+\alpha_2w_{(4)}^{(1)}=0,
\]
\[
v^{(1)}_{(3,1)}(y_1,y_2,z_1,z_2)=\Delta_{12}(w_{(3)}(y_1,z_1))z_1-3w_{(3)}(y_1,z_1)z_2
\]
\[
=-(2\alpha_1y_1+\alpha_2z_1)(y_1z_2-y_2z_1)z_1=-(2\alpha_1w_{(3,1)}^{(1)}+\alpha_2w_{(3,1)}^{(0)})=0.
\]
Similarly we obtain $w^{(2)}_{(4)}=y_1(\alpha_1y_1+\alpha_2z_2)(y_1z_1)=\alpha_1w_{(4)}^{(3)}+\alpha_2w_{(4)}^{(2)}=0$ and
\[
w^{(2)}_{(3,1)}=y_1(\alpha_1y_1+2\alpha_2z_1)(y_1z_2-y_2z_1)=\alpha_1w_{(3,1)}^{(2)}+2\alpha_2w_{(3,1)}^{(1)}=0
\]
by multiplication from the left.

{\rm (ii)} By substitution with elements of $W_d(2)$:
\[
v^{(3)}_{(4)}=\Delta_1(y_1z_1,w_{(3)}(y_1,z_1))=(\alpha_1y_1^2+2(\alpha_1+\alpha_2)y_1z_1+\alpha_2z_1^2)y_1z_1
\]
\[
=\alpha_1w_{(4)}^{(3)}+2(\alpha_1+\alpha_2)w_{(4)}^{(2)}+\alpha_2w_{(4)}^{(1)}=0,
\]
\[
v^{(3)}_{(3,1)}=(\alpha_1y_1^2-\alpha_2z_1^2)(y_1z_2-y_2z_1)=\alpha_1w_{(3,1)}^{(2)}-\alpha_2w_{(3,1)}^{(0)}=0,
\]
\[
v^{(1)}_{(2^2)}=-(\alpha_1+\alpha_2)(y_1z_2-y_2z_1)^2=-(\alpha_1+\alpha_2)w_{(2^2)}^{(0)}=0.
\]

{\rm (iii)} By substitution with elements of $W_d(1^2)$:
\[
v^{(4)}_{(3,1)}=\alpha_1w_{(3,1)}^{(2)}+2(\alpha_1+\alpha_2)w_{(3,1)}^{(1)}+\alpha_2w_{(3,1)}^{(0)}=0.
\]
\end{lemma}

\begin{proposition}\label{chi4 for (3)}
The cocharacter $\chi_4({\mathfrak U}_{\alpha})$ of the variety of bicommutative algebras ${\mathfrak U}_{\alpha}$ defined by the identity {\rm (\ref{PI (3)})} is
\[
\chi_4({\mathfrak U}_{\alpha})=\begin{cases}
0,\text{ if }\alpha_1\alpha_2(\alpha_1+\alpha_2)(\alpha_1-\alpha_2)\not=0;\\
\chi_{(4)}+\chi_{(3,1)},\text{ if }\alpha_1\alpha_2=0;\\
\chi_{(4)}+\chi_{(2^2)},\text{ if }\alpha_1+\alpha_2=0;\\
\chi_{(3,1)}, \text{ if }\alpha_1-\alpha_2=0.
\end{cases}
\]
\end{proposition}

\begin{proof}
The identities $v^{(i)}_{(4)}=0$, $i=1,2,3$, from Lemma \ref{consequences of M(3) in G} form the homogeneous linear system
\begin{equation}
\text{
\begin{tabular}{|rrrrl}\label{identities from (3) to (4)}
$v^{(1)}_{(4)}=$&&$\alpha_1w_{(4)}^{(2)}$&$+\alpha_2w_{(4)}^{(1)}$&=0\\
$v^{(2)}_{(4)}=$&$\alpha_1w_{(4)}^{(3)}$&$+\alpha_2w_{(4)}^{(2)}$&&=0\\
$v^{(3)}_{(4)}=$&$\alpha_1w_{(4)}^{(3)}$&$+2(\alpha_1+\alpha_2)w_{(4)}^{(2)}$&$+\alpha_2w_{(4)}^{(1)}$&=0\\
\end{tabular}
}
\end{equation}
The determinant of the matrix of the system (\ref{identities from (3) to (4)})
\[
\left(\begin{matrix}
0&\alpha_1&\alpha_2\\
\alpha_1&\alpha_2&0\\
\alpha_1&2(\alpha_1+\alpha_2)&\alpha_2
\end{matrix}\right)
\]
is equal to $\alpha_1\alpha_2(\alpha_1+\alpha_2)$. Hence, if $\alpha_1\alpha_2(\alpha_1+\alpha_2)\not=0$, then the system
(\ref{identities from (3) to (4)}) has the zero solution only. Hence $\chi_{(4)}$ does not participate in $\chi_4({\mathfrak U}_{\alpha})$.
If $\alpha_1\alpha_2(\alpha_1+\alpha_2)=0$, then the rank of the matrix is equal to 2, the system has one nonzero solution only
and $\chi_{(4)}$ participates with multiplicity 1 in $\chi_4({\mathfrak U}_{\alpha})$.

The case with identities $v^{(i)}_{(3,1)}=0$, $i=1,2,3,4$, is similar. We consider the linear system of equations
\begin{equation}
\text{
\begin{tabular}{|rrrrl}\label{identities from (3) to (3,1)}
$v^{(1)}_{(3,1)}=$&&$2\alpha_1w_{(3,1)}^{(1)}$&$+\alpha_2w_{(3,1)}^{(0)}$&=0\\
$v^{(2)}_{(3,1)}=$&$\alpha_1w_{(3,1)}^{(2)}$&$+2\alpha_2w_{(3,1)}^{(1)}$&&=0\\
$v^{(3)}_{(3,1)}=$&$\alpha_1w_{(3,1)}^{(2)}$&&$-\alpha_2w_{(3,1)}^{(0)}$&=0\\
$v^{(4)}_{(3,1)}=$&$\alpha_1w_{(3,1)}^{(2)}$&$+2(\alpha_1+\alpha_2)w_{(3,1)}^{(1)}$&$+\alpha_2w_{(3,1)}^{(0)}$&=0\\
\end{tabular}
}
\end{equation}
and want to determine the rank of its matrix
\[
\left(\begin{matrix}
0&2\alpha_1&\alpha_2\\
\alpha_1&2\alpha_2&0\\
\alpha_1&0&-\alpha_2\\
\alpha_1&2(\alpha_1+\alpha_2)&\alpha_2
\end{matrix}\right).
\]
The fourth row of the matrix is a sum of the first two and can be removed without changing the rank of the matrix.
Then the determinant of the matrix formed by the first three rows is equal to $2\alpha_1\alpha_2(\alpha_1-\alpha_2)$.
If $\alpha_1\alpha_2(\alpha_1-\alpha_2)\not=0$, then $\chi_{(3,1)}$ does not participate in $\chi_4({\mathfrak U}_{\alpha})$.
When $\alpha_1\alpha_2(\alpha_1-\alpha_2)=0$ the rank of the matrix is equal to 2 and
$\chi_{(3,1)}$ participates in $\chi_4({\mathfrak U}_{\alpha})$ with multiplicity 1.

Finally, by (\ref{multiplicities of cocharacters of B}), the multiplicity of $\chi_{(2^2)}$ in $\chi_4({\mathfrak U})_{\alpha}$
is $\leq 1$. The equation
\[
v^{(1)}_{(2^2)}=-(\alpha_1+\alpha_2)w_{(2^2)}^{(0)}=0
\]
shows that $\chi_{(2^2)}$ participates in $\chi_4({\mathfrak U}_{\alpha})$ if and only if $\alpha_1+\alpha_2=0$.
\end{proof}

\begin{corollary}\label{hwv in subvariety}
The highest weight vectors of degree $4$ in the free algebra of the variety ${\mathfrak U}_{\alpha}$ are:

{\rm (i)} $\alpha_1\alpha_2(\alpha_1+\alpha_2)(\alpha_1-\alpha_2)\not=0$: There are no nonzero polynomials of degree $4$;

{\rm (ii)} $\alpha_2=0$: $w_{(4)}^{(1)}$ and $w_{(3,1)}^{(0)}$;

{\rm (iii)} $\alpha_1=0$: $w_{(4)}^{(3)}$ and $w_{(3,1)}^{(2)}$;

{\rm (iv)} $\alpha_1+\alpha_2=0$: $w_{(4)}^{(1)}$ and $w_{(2^2)}^{(0)}$;

{\rm (v)} $\alpha_1-\alpha_2=0$: $w_{(3,1)}^{(2)}$.
\end{corollary}

\begin{proof}
Depending on $\alpha_1$ and $\alpha_2$, it follows from the system (\ref{identities from (3) to (4)}) that all $w_{(4)}^{(i)}$ are proportional
to the highest weight vector given in the statement of the Corollary.
Similarly, the system (\ref{identities from (3) to (4)}) answers the case $\lambda=(3,1)$.
Finally, for $\lambda=(2^2)$ there is only one submodule $W_d(\lambda)$ in $G_d({\mathfrak U}_{\alpha})$ and it is generated by $w_{(2^2)}^{(0)}$.
\end{proof}

\subsection{$S_n$-module structure}
In this subsection we shall describe the $S_n$-module structure of $P_n({\mathfrak U}_{\alpha})$.
It is clear that
\[
P_1({\mathfrak U}_{\alpha})=M(1),\,P_2({\mathfrak U}_{\alpha})=M(2)+M(1^2),\,P_3({\mathfrak U}_{\alpha})=M(3)+2M(2,1)
\]
and it is sufficient to handle the cases $n\geq 3$.

{\bf 1. Case $\bm{\alpha_1\alpha_2(\alpha_1+\alpha_2)(\alpha_1-\alpha_2)\not=0}$.}
By Proposition \ref{chi4 for (3)} $\chi_4({\mathfrak U}_{\alpha})=0$ and this immediately implies that $\chi_n({\mathfrak U}_{\alpha})=0$ for all $n\geq 4$.

{\bf 2. Case $\bm{\alpha_2=0}$.} Then the variety ${\mathfrak U}_{\alpha}$ is defined by the identity $x(xx)=0$, i.e. $y_1(y_1z_1)=0$ in $G_d({\mathfrak U}_{\alpha})$.

\begin{proposition}\label{chi for alpha2=0}
When $\alpha_2=0$ we have
\[
P_n({\mathfrak U}_{\alpha})=M(n)+M(n-1,1),\; n\geq 4.
\]
The corresponding irreducible $GL_d(K)$-submodules $W_d(n)$ and $W_d(n-1,1)$ in $G_d({\mathfrak U}_{\alpha})$ are generated, respectively, by
$w_{(n)}^{(1)}$ and $w_{(n-1,1)}^{(0)}$.
\end{proposition}

\begin{proof}
By Corollary \ref{bounds for cocharacters} it is sufficient to consider the cases $\lambda=(\lambda_1,\lambda_2)$ with $\lambda_2=0,1,2$.
The identities (\ref{identities from (3) to (4)}) imply $w_{(4)}^{(3)}=w_{(4)}^{(2)}=0$
and the $GL_d(K)$-submodule $W_d(4)$ in $G_d({\mathfrak U}_{\alpha})$ is generated by $w_{(4)}^{(1)}=y_1z_1^3$.
Similarly, for $n\geq 5$ and $j\geq 2$ any $w_{(n)}^{(j)}$ can be obtained by multiplication from both sizes of $w_{(3)}^{(2)}$.
Hence $W_d(n)\subset G_d({\mathfrak U}_{\alpha})$, $n\geq 5$, is generated by $w_{(n)}^{(1)}$.
Similarly, by (\ref{identities from (3) to (3,1)}) $w_{(3,1)}^{(2)}=w_{(3,1)}^{(1)}=0$, $w_{(3,1)}^{(0)}\not=0$ and we obtain that
$w_{(n-1,1)}^{(j)}=0$ in $G_d({\mathfrak U}_{\alpha})$ for all $n\geq 5$ and $j\geq 1$. Hence $W_d(n-1,1)\subset G_d({\mathfrak U}_{\alpha})$, $n\geq 5$, is generated by $w_{(n-1,1)}^{(0)}$.
It is easy to see that $w_{(n)}^{(1)}$ and $w_{(n-1,1)}^{(0)}$ are nonzero in $G_d({\mathfrak U}_{\alpha})$ because $\deg_{Y_d}(w_{(n)}^{(1)})=\deg_{Y_d}(w_{(n-1,1)}^{(0)})=1$
and for all consequences of $y_1^2z_1=0$ in $G_d({\mathfrak U}_{\alpha})$ we have $\deg_{Y_d}\geq 2$.
Finally, $w_{(2^2)}^{(0)}=0$ in $G_d({\mathfrak U}_{\alpha})$ implies $w_{(n-2,2)}^{(j)}=0$ for all $n\geq 4$.
\end{proof}

{\bf 3. Case $\bm{\alpha_1=0}$.} This case is similar to the case $\alpha_2=0$.

{\bf 4. Case $\bm{\alpha_1+\alpha_2=0}$.} Then the variety ${\mathfrak U}_{\alpha}$ is defined by the identity
\[
x(xx)-(xx)x=0\;(y_1^2z_1-y_1z_1^2=0\text{ in }G_d({\mathfrak U}_{\alpha})).
\]

\begin{proposition}\label{chi for alpha1+alpha2=0}
When $\alpha_1+\alpha_2=0$ we have
\[
P_4({\mathfrak U}_{\alpha})=M(4)+M(2^2),\; P_n({\mathfrak U}_{\alpha})=M(n),\,n\geq 5.
\]
The corresponding irreducible $GL_d(K)$-submodule $W_d(n)$ in $G_d({\mathfrak U}_{\alpha})$, $n\geq 4$, is generated by $w_{(n)}^{(1)}$.
\end{proposition}

\begin{proof}
The equality $P_4({\mathfrak U}_{\alpha})=M(4)+M(2^2)$ follows from Proposition \ref{chi4 for (3)}.
The equality $y_1^2z_1=y_1z_1^2$ in $G_d({\mathfrak U}_{\alpha})$ implies that
\[
y_1^pz_1^{n-p}=y_1^qz_1^{n-q},\,n\geq 3,\,p,q=1,\ldots,n-1.
\]
On the other hand, the one-dimensional algebra $A$ with basis $\{a\}$ and multiplication $a^2=a$ belongs to ${\mathfrak U}_{\alpha}$ and
the image of $\varphi(y_1z_1^{n-1})=a\not=0$ under the homomorphism $\varphi:G_d({\mathfrak U}_{\alpha})\to A$ which sends the monomials of $G_d({\mathfrak U}_{\alpha})$ to $a$ shows that
$\chi_{(n)}$ participates with multiplicity 1 in $\chi_n({\mathfrak U}_{\alpha})$.
Since $\chi_{(3,1)}$ does not participate in $\chi_4({\mathfrak U}_{\alpha})$, this means that
$y_1^pz_1^{2-p}(y_1z_2-y_2z_1)=0$ in $G_d({\mathfrak U}_{\alpha})$ for $p=0,1,2$, we have $y_1^qz_1^r(y_1z_2-y_2z_1)=0$ for all $q+r\geq 2$
and $\chi_{(n-1,1)}$ does not participate in $\chi_p({\mathfrak U}_{\alpha})$ for $n\geq 4$. To complete the proof it is sufficient to show that
$w_{(3,2)}^{(j)}=0$ in $G_d({\mathfrak U}_{\alpha})$ for $j=0,1$, i.e. $y_1(y_1z_2-y_2z_1)^2=(y_1z_2-y_2z_1)^2z_1=0$.
Since $w_{(3,1)}^{(i)}=y_1^iz_1^{2-i}(y_1z_2-y_2z_1)=0$, we obtain
\[
w_{(3,2)}^{(1)}=y_1(y_1z_2-y_2z_1)^2=y_1(y_1z_2-y_2z_1)y_1\cdot z_2-y_1(y_1z_2-y_2z_1)z_1\cdot y_2=0
\]
and similarly $w_{(3,2)}^{(0)}=0$.
\end{proof}

{\bf 5. Case $\bm{\alpha_1-\alpha_2=0}$.} In this case $y_1^2z_1=-y_1z_1^2$ and
$\chi_4({\mathfrak U)}_{\alpha}=\chi_{(3,1)}$.

\begin{proposition}\label{chi for alpha1-alpha2=0}
When $\alpha_1-\alpha_2=0$ we have $\chi_n({\mathfrak U}_{\alpha})=0$ for $n\geq 5$.
\end{proposition}

\begin{proof}
Since $w_{(4)}^{(i)}=0$, $i=1,2,3$, and $w_{(2^2)}^{(0)}=0$ we derive for $n\geq 5$ that
$w_{(n)}^{(i)}=0$, $i=1,\ldots,n-1$, and $w_{(n-2,2)}^{(j)}=0$, $j=0,1,\ldots,n-4$, i.e.
$\chi_{(n)}$ and $\chi_{(n-2,2)}$ do not participate in $\chi_n({\mathfrak U}_{\alpha})$.
By (\ref{identities from (3) to (3,1)}) we have
\[
2w_{(3,1)}^{(1)}+w_{(3,1)}^{(0)}=w_{(3,1)}^{(2)}+2w_{(3,1)}^{(1)}=0
\]
which after multiplication by $y_1$ and $z_1$ implies the linear system
\[
\begin{split}
2w_{(4,1)}^{(2)}+w_{(4,1)}^{(1)}=w_{(4,1)}^{(3)}+2w_{(4,1)}^{(2)}=0\\
2w_{(4,1)}^{(1)}+w_{(4,1)}^{(0)}=w_{(4,1)}^{(2)}+2w_{(4,1)}^{(1)}=0.
\end{split}
\]
The determinant of the matrix of the system
\[
\left(
\begin{matrix}
0&2&1&0\\
1&2&0&0\\
0&0&2&1\\
0&1&2&0
\end{matrix}
\right)
\]
is nonzero and hence $\chi_{(n-1,1)}$ does not participate in $\chi_n({\mathfrak U}_{\alpha})$.
Therefore $\chi_n({\mathfrak U}_{\alpha})=0$ for $n\geq 5$.
\end{proof}

\subsection{The lattice of subvarieties}
As in \cite{VDr} we shall describe the lattice of subvarieties of the variety ${\mathfrak U}_{\alpha}$ in the language of graph theory.
We shall define an oriented graph $\Gamma({\mathfrak U}_{\alpha})$.
The vertices of the graph are the irreducible $S_n$-submodules $M(\lambda)$, $\lambda\vdash n$.
Two vertices $M(\lambda)$, $\lambda\vdash n$, and $M(\mu)$, $\mu\vdash n+1$ are connected with an oriented edge
with beginning $M(\lambda)$ and ending $M(\mu)$ if the identities in $M(\mu)$ are consequences of the identities of $M(\lambda)$.
If $\mathfrak M$ is a subvariety of ${\mathfrak U}_{\alpha}$ with T-ideal $T({\mathfrak M})\subset F({\mathfrak U}_{\alpha})$,
then the graph $\Gamma_0({\mathfrak M})$ is the subgraph of $\Gamma({\mathfrak U}_{\alpha})$ with vertices corresponding to the $S_n$-submodules $M(\lambda)\subset T({\mathfrak M})$.
Then the subvarieties of ${\mathfrak U}_{\alpha}$ are in one-to-one correspondence with the subgraphs of $\Gamma({\mathfrak U}_{\alpha})$ and
\[
\Gamma_0(\text{\rm var}({\mathfrak M}_1\cup{\mathfrak M}_2))=\Gamma_0({\mathfrak M}_1)\cap \Gamma_0({\mathfrak M}_2),\;
\Gamma_0({\mathfrak M}_1\cap{\mathfrak M}_2)=\Gamma_0({\mathfrak M}_1)\cup \Gamma_0({\mathfrak M}_2).
\]

In the description of the subvarieties of ${\mathfrak U}_{\alpha}$ we use the highest weight vectors given in Corollary \ref{hwv in subvariety}
and the propositions in the previous subsection for the $S_n$-module structure of $P_n({\mathfrak U}_{\alpha})$.

{\bf 1. Case $\bm{\alpha_1\alpha_2(\alpha_1+\alpha_2)(\alpha_1-\alpha_2)\not=0}$.}

\begin{theorem}\label{lattice case 1}
When $\alpha_1\alpha_2(\alpha_1+\alpha_2)(\alpha_1-\alpha_2)\not=0$ the consequences in $G_d({\mathfrak U}_{\alpha})$ of degree $n+1$ of the polynomial $w_{\lambda}^{(n)}$, $\lambda\vdash n$,
are equivalent to the following polynomials:

{\rm (i)} For $w_{(1)}=x_1$: $w_{(2)}^{(1)}=y_1z_1$, $w_{(1^2)}^{(0)}=y_1z_2-y_2z_1$;

{\rm (ii)} For $w_{(2)}^{(1)}=y_1z_1$: $w_{(3)}^{(2)}=y_1^2z_1$, $w_{(2,1)}^{(1)}$, $w_{(2,1)}^{(0)}$;

{\rm (iii)} For $w_{(1^2)}^{(0)}=y_1z_2-y_2z_1$: $w_{(3)}^{(2)}=y_1^2z_1$, $w_{(2,1)}^{(1)}$, $w_{(2,1)}^{(0)}$.

Hence an identity of degree $1$ or $2$ implies, respectively, all polynomials of degree $2$ and $3$ in $G_d({\mathfrak U}_{\alpha})$.

{\rm (iv)} The polynomials $w_{(3)}^{(2)}$ and $\beta_1w_{(2,1)}^{(1)}+\beta_2w_{(2,1)}^{(0)}$, $(\beta_1,\beta_2)\not=(0,0)$, do not have consequences of degree $4$.
\end{theorem}
\vskip-2truecm
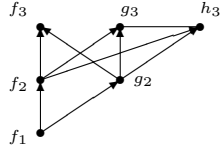
\[
\begin{picture}(80,110)
\put(0,0){\circle*{3}} \put(0,20){\circle*{3}}
\put(0,40){\circle*{3}}

\put(30,20){\circle*{3}} \put(30,40){\circle*{3}}
\put(60,40){\circle*{3}}

\put(0,0){\vector(0,2){20}}\put(-12,-3){\tiny{$f_1$}}
\put(0,20){\vector(0,2){20}}\put(-12,17){\tiny{$f_2$}}
\put(-12,45){\tiny{$f_3$}}

\put(0,0){\vector(3,2){30}}\put(35,19){\tiny{$g_2$}}
\put(30,45){\tiny{$g_3$}}
\put(60,45){\tiny{$h_3$}}

\put(0,20){\vector(3,2){30}}
\put(0,20){\line(3,1){30}} \put(30,30){\vector(3,1){30}}
\put(30,20){\vector(0,2){20}} \put(30,20){\vector(-3,2){30}}
\put(30,20){\vector(3,2){30}}
\put(30,40){\line(3,0){30}}

\put(-90,-20){\tiny{$f_1=\omega_{(1)}, f_2=\omega_{(2)}^{(1)},
f_3=\omega_{(3)}^{(2)}, g_2=\omega_{(1^2)}^{(0)}, g_3=\omega_{(2,1)}^{(1)}, h_3=\omega_{(2,1)}^{(0)}$}}

\put(-10,-35){\tiny{$\lambda=(3)$, Case 1}}
\end{picture}
\]
\vskip1.5truecm
\begin{proof}
By Lemma \ref{hwv of G} the highest weight vectors of the $GL_d(K)$-submodules $W_d(\lambda)$ of $G_d({\mathfrak U}_{\alpha})$, $\lambda\vdash n$, $n=1,2,3$, are generated by:

$n=1$: $w_{(1)}$;

$n=2$: $w_{(2)}^{(1)}$, $w_{(1^2)}^{(0)}$;

$n=3$: $w_{(3)}^{(2)}$ (and $w_{(3)}^{(1)}$ is proportional to $w_{(3)}^{(2)}$ because $\alpha_2\not=0$), $\beta_1w_{(2,1)}^{(1)}+\beta_2w_{(2,1)}^{(0)}$, $(\beta_1,\beta_2)\not=(0,0)$.

It is obvious that all homogeneous polynomials of degree 2 in $G_d({\mathfrak U}_{\alpha})$ follow from $w_{(1)}$.
Proposition \ref{chi4 for (3)} shows that the homogeneous polynomials of degree 3 do not have nonzero consequences of degree 4.

For the consequences of $w_{(2)}^{(1)}=y_1z_1$:
\[
w_{(3)}^{(2)}=y_1w_{(2)}^{(1)}(y_1,z_1),\, w_{(3)}^{(1)}=w_{(2)}^{(1)}(y_1,z_1)z_1,
\]
\[
w_{(2,1)}^{(1)}=y_1\Delta_{12}(w_{(2)}^{(1)}(y_1,z_1)-2y_2w_{(2)}^{(1)}(y_1,z_1)),
\]
\[
w_{(2,1)}^{(0)}=-\Delta_{12}(w_{(2)}^{(1)}(y_1,z_1)z_1+2w_{(2)}^{(1)}(y_1,z_1))z_2.
\]

For the consequences of $w_{(1^2)}^{(0)}=y_1z_2-y_2z_1$:
\[
w_{(1^2)}^{(0)}(y_1,y_1z_1,z_1,y_1z_1)=w_{(3)}^{(2)}-w_{(3)}^{(1)}.
\]
Since $\alpha_1w_{(3)}^{(2)}+\alpha_2w_{(3)}^{(1)}=0$ in $G_d({\mathfrak U}_{\alpha})$ and $\alpha_1+\alpha_2\not=0$, we obtain that both
$w_{(3)}^{(2)}$ and $w_{(3)}^{(1)}$ follow from $w_{(1^2)}^{(0)}$.
\[
w_{(2,1)}^{(1)}=y_1w_{(1^2)}^{(0)}, w_{(2,1)}^{(0)}=w_{(1^2)}^{(0)}z_1
\]
and $w_{(2,1)}^{(1)}$ and $w_{(2,1)}^{(0)}$ also follow from $w_{(1^2)}^{(0)}$.
\end{proof}

{\bf 2. Case $\bm{\alpha_2=0}$.} Then $w_{(3)}^{(2)}=y_1(y_1z_1)=0$ in $G_d({\mathfrak U}_{\alpha})$.

\begin{theorem}\label{lattice case 2}
When $\lambda_2=0$ the consequences in $G_d({\mathfrak U}_{\alpha})$ of degree $n+1$ of the polynomial $w_{\lambda}^{(n)}$, $\lambda\vdash n$, are equivalent to the following polynomials:

{\rm (i)} An identity of degree $1$ or $2$ implies, respectively, all polynomials of degree $2$ and $3$ in $G_d({\mathfrak U}_{\alpha})$.

{\rm (ii)} The identities $w_{(3)}^{(1)}$ and $\beta_1w_{(2,1)}^{(1)}+\beta_2w_{(2,1)}^{(0)}$, $\beta_2\not=0$, imply $w_{(4)}^{(1)}$ and $w_{(3,1)}^{(0)}$, i.e. all polynomials of degree $4$.

{\rm (iii)} For $n\geq 4$ the identities $w_{(n)}^{(1)}$ and $w_{(n-1,1)}^{(0)}$ imply $w_{(n+1)}^{(1)}$ and $w_{(n,1)}^{(0)}$, i.e. all polynomials of degree $n+1$.

{\rm (iv)} The polynomial $w_{(2,1)}^{(1)}$ does not have consequences of degree $4$.
\end{theorem}
\vskip-2truecm
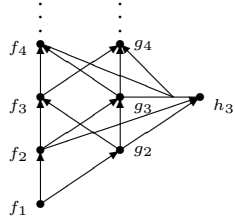
\[
\begin{picture}(80,130)
\put(0,0){\circle*{3}} \put(0,20){\circle*{3}}
\put(0,40){\circle*{3}} \put(0,60){\circle*{3}}
\put(0,65){\circle*{1}} \put(0,70){\circle*{1}}
\put(0,75){\circle*{1}}

\put(30,20){\circle*{3}} \put(30,40){\circle*{3}}
\put(30,60){\circle*{3}}
\put(30,65){\circle*{1}} \put(30,70){\circle*{1}}
\put(30,75){\circle*{1}}

\put(60,40){\circle*{3}}

\put(0,0){\vector(0,2){20}}\put(-12,-3){\tiny{$f_1$}}
\put(0,20){\vector(0,2){20}}\put(-12,17){\tiny{$f_2$}}
\put(0,40){\vector(0,2){20}}\put(-12,35){\tiny{$f_3$}}
\put(-12,57){\tiny{$f_4$}}

\put(0,0){\vector(3,2){30}}\put(35,19){\tiny{$g_2$}}
\put(30,20){\vector(0,2){20}}\put(35,35){\tiny{$g_3$}}
\put(30,40){\vector(0,2){20}}\put(35,59){\tiny{$g_4$}}

\put(30,40){\line(3,0){30}}

\put(30,20){\vector(3,2){30}}\put(65,35){\tiny{$h_3$}}
\put(0,20){\vector(3,2){30}} \put(0,40){\vector(3,2){30}}
\put(30,20){\vector(-3,2){30}} \put(30,40){\vector(-3,2){30}}

\put(50,40){\vector(-1,1){20}}
\put(50,40){\line(-5,2){50}}

\put(0,20){\line(3,1){30}} \put(30,30){\vector(3,1){30}}

\put(-90,-20){\tiny{$f_1=\omega_{(1)}, f_n=\omega_{(n)}^{(1)}, g_n=\omega_{(n-1,1)}^{(0)}, n\geq 2, h_3=\omega_{(2,1)}^{(1)}$}}

\put(-10,-35){\tiny{$\lambda=(3)$, Case 2}}
\end{picture}
\]
\vskip1.5truecm
\begin{proof}
The proofs of the following cases repeat verbatim the proofs in Theorem \ref{lattice case 1}:

The consequences of $w_{(1)}$ and $w_{(2)}^{(1)}$;

The fact that $w_{(1^2)}^{(0)}$ implies $w_{(2,1)}^{(1)}$ and $w_{(2,1)}^{(1)}$.

By multiplication from the right by $z_1$ we obtain that $w_{(n)}^{(1)}$ and $w_{(n-1,1)}^{(0)}$, $n\geq 4$,
imply, respectively, $w_{(n+1)}^{(1)}$ and $w_{(n,1)}^{(0)}$.

Since $y_1^2z_1=0$ we have $\Delta_1(y_1z_1,w_{(1^2)}^{(0)})=y_1^2z_1-y_1z_1^2=-y_1z_1^2=-w_{(3)}^{(1)}$
which completes the case $w_{(1^2)}^{(0)}$.

Let $w_{(n)}^{(1)}=y_1z_1^{n-1}=0$, $n\geq 3$. Then
\[
0=\Delta_{12}(w_{(n)}^{(1)})z_1=(y_2z_1^{n-1}+(n-1)y_1z_1^{n-2}z_2)z_1=y_2z_1^n-(n-1)y_1z_1^{n-1}z_2,
\]
$y_2z_1^n=0$, and hence $w_{(n,1)}^{(0)}=(y_1z_2-y_2z_1)z_1^{n-1}=y_1z_1^{n-1}z_2-y_2z_1^n=0$.

Similarly $w_{(n-1,1)}^{(0)}=0$, $n\geq 3$, implies
\[
0=\Delta_2(y_1z_1,w_{(n-1,1)}^{(0)})=y_1^2z_1^{n-1}-y_1z_1^n=-w_{(n+1)}^{(1)}.
\]
Since $w_{(2,1)}^{(1)})$ does not contain summands which are linear in $Y_d$, all consequences of higher degree have the same property.
Hence $\beta_1w_{(2,1)}^{(1)}+\beta_2w_{(2,1)}^{(0)}$ and $\beta_2w_{(2,1)}^{(0)}$ have the same consequences of degree 4 and this completes the proof.
\end{proof}

{\bf 3. Case $\bm{\alpha_1=0}$.} Then $w_{(3)}^{(1)}=(y_1z_1)z_1=0$ in $G_d({\mathfrak U}_{\alpha})$.
This case is similar to the case $\alpha_2=0$.

{\bf 4. Case $\bm{\alpha_1+\alpha_2=0}$.} Then $w_{(3)}^{(2)}-w_{(3)}^{(1)}=y_1^2z_1-y_1z_1^2=0$ in $G_d({\mathfrak U}_{\alpha})$.

\begin{theorem}\label{lattice case 4}
When $\alpha_1+\alpha_2=0$ the consequences in $G_d({\mathfrak U}_{\alpha})$ of degree $n+1$ of the polynomial $w_{\lambda}^{(n)}$, $\lambda\vdash n$, are equivalent to the following polynomials:

{\rm (i)} The identity $w_{(n)}^{(1)}$, $n \geq 1$ implies all polynomials of degree $n+1$.

{\rm (ii)} The identities $w_{(1^2)}^{(0)}$ and $\beta_1w_{(2,1)}^{(1)}+\beta_2w_{(2,1)}^{(0)}$, $(\beta_1,\beta_2)\not=(0,0)$, imply, respectively,
$w_{(2,1)}^{(1)},w_{(2,1)}^{(0)}$ and $w_{(2^2)}^{(0)}$.

{\rm (iii)} $w_{(2^2)}^{(0)}$ does not have consequences of higher degree.
\end{theorem}
\vskip-1.5truecm
\[
\begin{picture}(80,140)
\put(0,0){\circle*{3}} \put(0,20){\circle*{3}}
\put(0,40){\circle*{3}} \put(0,60){\circle*{3}}
\put(0,80){\circle*{3}}
\put(0,85){\circle*{1}} \put(0,90){\circle*{1}}
\put(0,95){\circle*{1}}

\put(30,20){\circle*{3}} \put(30,40){\circle*{3}}
\put(30,60){\circle*{3}}

\put(60,40){\circle*{3}}

\put(0,0){\vector(0,2){20}}\put(-12,-3){\tiny{$f_1$}}
\put(0,20){\vector(0,2){20}}\put(-12,17){\tiny{$f_2$}}
\put(0,40){\vector(0,2){20}}\put(-12,35){\tiny{$f_3$}}
\put(0,60){\vector(0,2){20}}\put(-12,57){\tiny{$f_4$}}
\put(-12,77){\tiny{$f_5$}}

\put(0,0){\vector(3,2){30}}\put(35,19){\tiny{$g_2$}}
\put(30,20){\vector(0,2){20}}\put(34,35){\tiny{$g_3$}}
\put(30,40){\vector(0,2){20}}\put(35,59){\tiny{$p_4$}}

\put(30,40){\line(3,0){30}}

\put(40,40){\vector(-1,2){10}}

\put(30,20){\vector(3,2){30}}\put(65,35){\tiny{$h_3$}}
\put(0,20){\vector(3,2){30}} \put(0,40){\vector(3,2){30}}
\put(60,40){\vector(-3,2){30}}

\put(0,20){\line(3,1){30}} \put(30,30){\vector(3,1){30}}

\put(-90,-20){\tiny{$f_1=\omega_{(1)}, f_n=\omega_{(n)}^{(1)}, g_2=\omega_{(1^2)}^{(0)}, g_3=\omega_{(2,1)}^{(1)},  h_3=\omega_{(2,1)}^{(0)}, p_4=\omega_{(2^2)}^{(0)}$}}

\put(-10,-35){\tiny{$\lambda= (3)$, Case 4}}
\end{picture}
\]
\vskip1.5truecm
\begin{proof}
The proof that $w_{(1)}$ and $w_{(2)}^{(1)}$ imply all polynomials of degree 2 and 3, respectively, is the same as in the previous cases.
It is also clear that $w_{(1^2)}^{(0)}$ implies $w_{(2,1)}^{(1)}$ and $w_{(2,1)}^{(0)}$.
The one-dimensional algebra $A$ with basis $\{a\}$ and multiplication $a^2=a$ satisfies the identity (\ref{PI (3)})
for $\alpha_1+\alpha_2=0$ and all identities $w_{(\lambda_1,\lambda_2)}^{(i)}=0$
with $\lambda_2>0$ and $w_{(n)}^{(1)}$, $n>1$ is different from 0 in $A$. Hence $w_{(1^2)}^{(0)}$, $w_{(2,1)}^{(1)}$, $w_{(2,1)}^{(0)}$ and $w_{(2^2)}^{(0)}$
do not imply any $w_{(n)}^{(1)}$. Clearly, $w_{(n)}^{(1)}$ implies $w_{(n+1)}^{(1)}$ and the only proof we need is that
$w_{(3)}^{(1)}$ and $\beta_1w_{(2,1)}^{(1)}+\beta_2w_{(2,1)}^{(0)}$ imply $w_{(2^2)}^{(0)}$.

If $w_{(3)}^{(1)}=0$ the condition $\alpha_1+\alpha_2=0$ in (\ref{PI (3)}) implies that $w_{(3)}^{(2)}=0$ and
as in Theorem \ref{lattice case 1} we obtain that all polynomials of degree 4 follow from $w_{(3)}^{(1)}=0$.

Direct computation shows that
\[
(\beta_1w_{(2,1)}^{(1)}(y_1,y_2,z_1,z_2)+\beta_2w_{(2,1)}^{(0)}(y_1,y_2,z_1,z_2))(\gamma_1y_2+\gamma_2z_2)
\]
\[
+(\beta_1w_{(2,1)}^{(1)}(y_2,y_1,z_2,z_1)+\beta_2w_{(2,1)}^{(0)}(y_2,y_1,z_2,z_1))(\gamma_1y_1+\gamma_2z_1)
\]
\[
=(\beta_1\gamma_2-\beta_2\gamma_1)(y_1z_2-y_2z_1)^2.
\]
For any $(\beta_1,\beta_2)\not=(0,0)$ we can choose $(\gamma_1,\gamma_2)$ such that $\beta_1\gamma_2-\beta_2\gamma_1\not=0$
and this implies that $w_{(2^2)}^{(0)}$ follows from $\beta_1w_{(2,1)}^{(1)}+\beta_2w_{(2,1)}^{(0)}$, $(\beta_1,\beta_2)\not=(0,0)$.
\end{proof}

{\bf 5. Case $\bm{\alpha_1-\alpha_2=0}$.} Hence $w_{(3)}^{(2)}+w_{(3)}^{(1)}=y_1^2z_1+y_1z_1^2=0$ in $G_d({\mathfrak U}_{\alpha})$.

\begin{theorem}\label{lattice case 5}
When $\alpha_1=\alpha_2$ the consequences in $G_d({\mathfrak U}_{\alpha})$ of degree $n+1$ of the polynomial $w_{\lambda}^{(n)}$, $\lambda\vdash n$, are equivalent to the following polynomials:

{\rm (i)} Any identity of degree $n=1,2,3$, implies all polynomials of degree $n+1$.

{\rm (ii)} $w_{(3,1)}^{(2)}$ does not have consequences of higher degree.
\end{theorem}
\vskip-2.5truecm
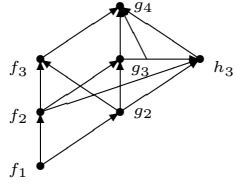
\[
\begin{picture}(80,140)
\put(0,0){\circle*{3}} \put(0,20){\circle*{3}}
\put(0,40){\circle*{3}} 

\put(30,20){\circle*{3}} \put(30,40){\circle*{3}}
\put(30,60){\circle*{3}}

\put(60,40){\circle*{3}}

\put(0,0){\vector(0,2){20}}\put(-12,-3){\tiny{$f_1$}}
\put(0,20){\vector(0,2){20}}\put(-12,17){\tiny{$f_2$}}
\put(-12,35){\tiny{$f_3$}}

\put(0,0){\vector(3,2){30}}\put(35,19){\tiny{$g_2$}}
\put(30,20){\vector(0,2){20}}\put(34,35){\tiny{$g_3$}}
\put(30,40){\vector(0,2){20}}\put(35,59){\tiny{$g_4$}}

\put(30,40){\vector(3,0){30}}

\put(30,20){\vector(3,2){30}}\put(65,35){\tiny{$h_3$}}
\put(0,20){\vector(3,2){30}} \put(0,40){\vector(3,2){30}}
\put(30,20){\vector(-3,2){30}} 
\put(60,40){\vector(-3,2){30}}

\put(40,40){\vector(-1,2){10}}

\put(0,20){\line(3,1){30}} \put(30,30){\vector(3,1){30}}

\put(-90,-20){\tiny{$f_1=\omega_{(1)}, f_2=\omega_{(2)}^{(1)}, g_2=\omega_{(1^2)}^{(0)}, f_3=\omega_{(3)}^{(2)},  g_3=\omega_{(2,1)}^{(1)},  h_3=\omega_{(2,1)}^{(0)}, g_4=\omega_{(3,1)}^{(2)}$}}

\put(-10,-35){\tiny{$\lambda= (3)$, Case 5}}
\end{picture}
\]
\vskip1.5truecm
\begin{proof}
The proof for the identities $w_{(1)}, w_{(2)}^{(1)}, w_{(3)}^{(2)}$ is the same as in the previous cases.
If $w_{(1^2)}^{(0)}=0$, then
\[
\Delta_2(y_1z_1,w_{(1^2)}^{(0)})=y_1^2z_1-y_1z_1^2=w_{(3)}^{(2)}-w_{(3)}^{(1)}=0
\]
and this together with the identity $\alpha_1(w_{(3)}^{(2)}+w_{(3)}^{(1)})=0$
implies $w_{(3)}^{(2)}=0$.
\end{proof}

\section{The variety defined by the identity $\beta_1x(xy-yx)+\beta_2(xy-yx)x=0$}

In this section we shall describe all subvarieties of the variety of bicommutative algebras ${\mathfrak V}_{\beta}$ determined by the polynomial identity (\ref{PI (2,1)}).
As in the previous section we shall work in the algebra $G_d({\mathfrak V}_{\beta})$ which is isomorphic to the relatively free algebra $F_d({\mathfrak V}_{\beta})$.
By (\ref{PI (2,1) in G}) the identity (\ref{PI (2,1)}) becomes
\[
w_{(2,1)}(y_1,y_2,z_1,z_2)=(\beta_1y_1+\beta_2z_1)(y_1z_2-y_2z_1)=\beta_1w_{(2,1)}^{(1)}+\beta_2w_{(2,1)}^{(0)}=0.
\]
We shall apply the results in Example \ref{consequences of M(2,1) in associative case} to find its consequences.
The linearization of $w_{(2,1)}$ is
\[
p(y_1,y_2,y_3,z_1,z_2,z_3)=(\beta_1y_1+\beta_2z_1)(y_2z_3-y_3z_2)+(\beta_1y_2+\beta_2z_2)(y_1z_3-y_3z_1)=0.
\]
The generators of the $S_2$-submodules $M(2)$ and $M(1^2)$ of $M(2,1)\downarrow S_2$ are, respectively,
$p_{(2)}(y_1,y_2,y_3,z_1,z_2,z_3)=p(y_1,y_2,y_3,z_1,z_2,z_3)$ and
\[
p_{(1^2)}(y_1,y_2,y_3,z_1,z_2,z_3)=\frac{1}{3}(p(y_1,y_2,y_3,z_1,z_2,z_3)+2p(y_1,y_3,y_2,z_1,z_3,z_2))
\]
\[
=(\beta_1y_3+\beta_2z_3)(y_1z_2-y_2z_1).
\]
Again, in the proof of the following lemma we use Criterion \ref{criterion hwv}.
By Corollary \ref{bounds for cocharacters} it is sufficient to consider the consequences in $M(\mu)$ for $\mu=(n),(n-1,1)$ only.

\begin{lemma}\label{consequences of M(2,1) in G}
All consequences of the identity {\rm (\ref{PI (2,1)})} which generate $GL_d(K)$-modules $W_d(4)$ and $W_d(3,1)$ in $G_d({\mathfrak V}_{\beta})$
are equivalent to the identities obtained in the following way:

{\rm (i)} By one side multiplication (first from the right):
\[
t^{(1)}_{(3,1)}(y_1,y_2,z_1,z_2)=w_{(2,1)}(y_1,y_2,z_1,z_2)z_1
\]
\[
=(\beta_1y_1+\beta_2z_1)(y_1z_2-y_2z_1)z_1=\beta_1w_{(3,1)}^{(1)}+\beta_2w_{(3,1)}^{(0)}=0,
\]
and similarly by multiplication from the left
\[
t^{(2)}_{(3,1)}(y_1,y_2,z_1,z_2)=y_1w_{(2,1)}(y_1,y_2,z_1,z_2)=\beta_1w_{(3,1)}^{(2)}+\beta_2w_{(3,1)}^{(1)}=0.
\]

{\rm (ii)} By substitution with elements of $W_d(2)$:
\[
t^{(4)}_{(1)}=w_{(2,1)}(y_1,y_1z_1,z_1,y_1z_1)=\beta_1w_{(4)}^{(3)}-(\beta_1-\beta_2)w_{(4)}^{(2)}-\beta_2w_{(4)}^{(1)}=0,
\]
\[
t^{(3)}_{(3,1)}=\frac{1}{2}p_{(2)}(y_1,y_1,y_1z_2+y_2z_1,z_1,z_1,y_1z_2+y_2z_1)+p_{(2)}(y_1,y_2,y_1z_1,z_1,z_2,y_1z_1)
\]
\[
=(\beta_1y_1^2+\beta_2z_1^2)(y_1z_2-y_2z_1)=\beta_1w^{(2)}_{(3,1)}+\beta_2w^{(0)}_{(3,1)}=0,
\]
\[
t^{(4)}_{(3,1)}=p_{(1^2)}(y_1,y_2,y_1z_1,z_1,z_2,y_1z_1)
\]
\[
=(\beta_1+\beta_2)y_1z_1(y_1z_2-y_2z_1)=(\beta_1+\beta_2)w^{(1)}_{(3,1)}=0.
\]

{\rm (iii)} By substitution with elements of $W_d(1^2)$:
\[
t^{(5)}_{(3,1)}=\frac{1}{2}p_{(2)}(y_1,y_2,y_1z_2-y_2z_1,z_1,z_2,y_1z_2-y_2z_1)
\]
\[
=(\beta_1y_1+\beta_2z_1)(y_1-z_1)(y_1z_2-y_2z_1)=\beta_1w^{(2)}_{(3,1)}-(\beta_1-\beta_2)w^{(1)}_{(3,1)}-\beta_2w^{(0)}_{(3,1)}=0.
\]
\end{lemma}

\begin{proposition}\label{chi4 for (2,1)}
The cocharacter $\chi_4({\mathfrak V}_{\beta})$ of the variety of bicommutative algebras ${\mathfrak V}_{\beta}$ defined by the identity {\rm (\ref{PI (2,1)})} is
\[
\chi_4({\mathfrak V}_{\beta})=\begin{cases}
2\chi_{(4)},\text{ if }\beta_1\beta_2(\beta_1+\beta_2)(\beta_1-\beta_2)\not=0\text{ or }\beta_1-\beta_2=0;\\
2\chi_{(4)}+\chi_{(3,1)},\text{ if }\beta_1\beta_2(\beta_1+\beta_2)=0.
\end{cases}
\]
\end{proposition}

\begin{proof}
By Lemma \ref{consequences of M(2,1) in G} the only consequence in $W_d(4)$ is
\[
t^{(4)}_{(1)}=\beta_1w_{(4)}^{(3)}-(\beta_1-\beta_2)w_{(4)}^{(2)}-\beta_2w_{(4)}^{(1)}=0.
\]
Since $\chi_4({\mathfrak B})=3\chi_{(4)}+2\chi_{(3,1)}$ we derive that $\chi_{(4)}$ participates with multiplicity 2 in $\chi_4({\mathfrak V}_{\beta})$.

We shall handle the case with identities $t^{(i)}_{(3,1)}=0$, $i=1,2,3,4,5$, as in the proof of Proposition \ref{chi4 for (3)}. We consider the linear system of equations
\begin{equation}
\text{
\begin{tabular}{|rrrrl}\label{identities from (2,1) to (3,1)}
$t^{(1)}_{(3,1)}=$&&$\beta_1w_{(3,1)}^{(1)}$&$+\beta_2w_{(3,1)}^{(0)}$&=0\\
$t^{(2)}_{(3,1)}=$&$\beta_1w_{(3,1)}^{(2)}$&$+\beta_2w_{(3,1)}^{(1)}$&&=0\\
$t^{(3)}_{(3,1)}=$&$\beta_1w_{(3,1)}^{(2)}$&&$+\beta_2w_{(3,1)}^{(0)}$&=0\\
$t^{(4)}_{(3,1)}=$&$$&$+(\beta_1+\beta_2)w_{(3,1)}^{(1)}$&$$&=0\\
$t^{(5)}_{(3,1)}=$&$\beta_1w_{(3,1)}^{(2)}$&$-(\beta_1-\beta_2)w_{(3,1)}^{(1)}$&$-\beta_2w_{(3,1)}^{(0)}$&=0\\
\end{tabular}
}
\end{equation}
Direct computations show that the rank of its matrix
\[
\left(\begin{matrix}
0&\beta_1&\beta_2\\
\beta_1&\beta_2&0\\
\beta_1&0&\beta_2\\
0&\beta_1+\beta_2&0\\
\beta_1&-(\beta_1-\beta_2)&-\beta_2
\end{matrix}\right)
\]
is equal to 3 if $\beta_1\beta_2(\beta_1+\beta_2)(\beta_1-\beta_2)\not=0$ or $\beta_1-\beta_2=0$ and
is equal to 2 if $\beta_1\beta_2(\beta_1+\beta_2)=0$ and this gives the multiplicity of $\chi_{(3,1)}$ in $\chi_4({\mathfrak V}_{\beta})$.
\end{proof}

\subsection{$S_n$-module structure}
We shall describe the $S_n$-module structure of $P_n({\mathfrak V}_{\beta})$.
The cases $n=1,2$ are clear and we shall consider the cases $n\geq 3$.

{\bf 1. Case $\bm{\beta_1\beta_2(\beta_1+\beta_2)(\beta_1-\beta_2)\not=0}$.}

\begin{proposition}\label{chi for general beta}
When $\beta_1\beta_2(\beta_1+\beta_2)(\beta_1-\beta_2)\not=0$ we have
\[
P_4({\mathfrak V}_{\beta})=2M(4),\;P_n({\mathfrak V}_{\beta})=M(n),\; n> 4.
\]
The corresponding irreducible $GL_d(K)$-submodules $W_d(2,1)$ and $W_d(n)$ in $G_d({\mathfrak V}_{\beta})$ are generated by
$w_{(2,1)}^{(1)}$ for $\lambda=(2,1)$, $w_{(n)}^{(n-1)}$ and $w_{(n)}^{(n-2)}$ for $n=3,4$ and of $w_{(n)}^{(n-1)}$ for $n>4$.
\end{proposition}

\begin{proof}
The statement of the proposition is clear for $n<4$. The identity $t^{(1)}_{(4)}=0$ from Lemma \ref{consequences of M(2,1) in G}  gives that
$w_{(4)}^{(1)}$ can be presented as a linear combination of $w_{(4)}^{(2)}$ and $w_{(4)}^{(3)}$ in $G_d({\mathfrak V}_{\beta})$.
Since there are no other identities $w_{(4)}=0$ in $G_d({\mathfrak V}_{\beta})$,
we obtain the case $n=4$. From the partial linearization $s_{(4)}(y_1,y_2,z_1,z_2)=\Delta_{12}(t^{(1)}_{(4)}(y_1,z_1))=0$ we obtain
$s_{(4)}(y_1,y_1z_1,z_1,y_1z_1)=0$. Together with $y_1t^{(1)}_{(4)}(y_1,z_1)=t^{(1)}_{(4)}(y_1,z_1)z_1=0$
we obtain a system of three equations in $w^{(i)}_{(5)}$, $i=1,2,3,4$:
\[
\text{
\begin{tabular}{|rrrrl}
$\beta_1w_{(4)}^{(4)}$&$-(\beta_1-\beta_2)w_{(4)}^{(3)}$&$-w_{(4)}^{(2)}$&&=0\\
&$\beta_1w_{(4)}^{(3)}$&$-(\beta_1-\beta_2)w_{(4)}^{(2)}$&$-w_{(4)}^{(1)}$&=0\\
$\beta_1w_{(4)}^{(4)}$&$+(\beta_1+2\beta_2)w_{(4)}^{(3)}$&$-(2\beta_1+\beta_2)w_{(4)}^{(2)}$&$-\beta_2w_{(4)}^{(1)}$&=0\\
\end{tabular}
}
\]
This system has a unique solution
\[
w^{(i)}_{(5)}=w^{(1)}_{(5)},\,i=2,3,4.
\]
From here we obtain that $w^{(i)}_{(n)}=w^{(n-1)}_{(n)}$ for all $n\geq 5$ and $i=1,2,\ldots,n-2$.
Since $w^{(n-1)}_{(n)}\not=0$ in the one-dimensional algebra $A$ with basis $\{a\}$ which we used before
and $A\in{\mathfrak V}_{\beta}$, we obtain that $w^{(n-1)}_{(n)}\not=0$ in $G_d({\mathfrak V}_{\beta})$.

Finally, by Proposition \ref{chi4 for (2,1)} $w^{(1)}_{(3,1)}=w^{(0)}_{(3,1)}=0$ in $G_d({\mathfrak V}_{\beta})$
and hence $w^{(i)}_{(n-1,1)}=0$ for all $n\geq 4$, $i=0,1,\ldots,n-2$.
\end{proof}

{\bf 2. Case $\bm{\beta_2=0}$.} The variety ${\mathfrak V}_{\beta}$ satisfies the identity $x_1[x_1,x_2]=0$ which in $G_d({\mathfrak V}_{\beta})$ becomes
\[
w_{(2,1)}(y_1,y_2,z_1,z_2)=y_1(y_1z_2-y_2z_1)=0.
\]

\begin{proposition}\label{chi for beta2=0}
When $\beta_2=0$ we have
\[
P_n({\mathfrak V}_{\beta})=2M(n)+M(n-1,1),\; n\geq 3.
\]
The corresponding irreducible $GL_d(K)$-submodules $W_d(n)$ and $W_d(n-1,1)$ in $G_d{\mathfrak V}_{\beta})$ in $G_d$ are generated by
$w_{(n)}^{(1)}$ and $w_{(n)}^{(2)}$ for $\lambda=(n)$ and of $w_{(n-1,1)}^{(0)}$ for $\lambda=(n-1,1)$.
\end{proposition}

\begin{proof}
Let
\[
s_{(2,1)}(y_1,y_2,y_3,z_1,z_2,z_3)=y_1(y_2z_3-y_3z_2)+y_2(y_1z_3-y_3z_1)=0
\]
be the partial linearization of $w_{(2,1)}=0$. All consequences of the form $t(y_1,z_1)=0$ are linear combinations of
polynomials obtained from $s_{(2,1)}=0$ after substitution of the variables with monomials and multiplication by monomials from both sides.
They are of the following form:
\[
y_1^ps_{(2,1)}(y_1^{a_1}z_1^{a_2},y_1,y_1,y_1^{a_1}z_1^{a_2},z_1,z_1)z_1^q,\; y_1^ps_{(2,1)}(y_1,y_1,y_1^{c_1}z_1^{c_2},z_1,z_1,y_1^{c_1}z_1^{c_2})z_1^q,
\]
\[
y_1^ps_{(2,1)}(y_1^{a_1}z_1^{a_2},y_1,y_1^{c_1}z_1^{c_2},y_1^{a_1}z_1^{a_2},z_1,y_1^{c_1}z_1^{c_2})z_1^q,
\]
\[
y_1^ps_{(2,1)}(y_1^{a_1}z_1^{a_2},y_1^{b_1}z_1^{b_2},y_1^{c_1}z_1^{c_2},y_1^{a_1}z_1^{a_2},y_1^{b_1}z_1^{b_2},y_1^{c_1}z_1^{c_2})z_1^q,
\]
$p,q\geq 0$, $a_1,a_2,b_1,b_2,c_1,c_2\geq 1$.
Direct computations show that
\[
y_1^ps_{(2,1)}(y_1^{a_1}z_1^{a_2},y_1,y_1,y_1^{a_1}z_1^{a_2},z_1,z_1)z_1^q=y_1^{p+a_1+1}z_1^{q+a_2}(z_1-y_1)=0.
\]
Since $a_1,a_2\geq 1$ we obtain that this expression is divisible by $y_1^2(z_1-y_1)z_1$ and we obtain the identities
\[
w_{(n)}^{(k+3)}(y_1,z_1)=w_{(n)}^{(k+2)}(y_1,z_1),\;k=0,1,\ldots,n-1.
\]
The same holds for the other consequences of the form $t(y_1,z_1)=0$. This means that the identities under considerations are equivalent to
\[
w_{(n)}^{(n-1)}=w_{(n)}^{(n-2)}=\cdots=w_{(n)}^{(2)}
\]
and the highest weight vectors of $W_d(n)\subset F_d({\mathfrak V}_{\beta})$ are linear combinations of the linearly independent elements
corresponding to $w_{(n)}^{(2)}$ and $w_{(n)}^{(1)}$.

Since $w_{(n-1,1)}^{(k)}=y_1^{k-1}w_{(2,1)}^{(1)}z_1^{n-k-2}$ for $k\geq 1$ we obtain that for $k\geq 1$ all $w_{(n-1,1)}^{(k)}$ follow from $w_{(2,1)}^{(1)}=0$.
Clearly, $w_{(n-1,1)}^{(0)}$ does not follow from $w_{(2,1)}^{(1)}=0$ because $\deg_y(w_{(n-1,1)}^{(0)})=1$, $\deg_y(w_{(2,1)}^{(1)})=2$
and the degree in $y$ does not decrease in the consequences. Hence $W_d(n-1,1)\subset F_d({\mathfrak V}_{\beta}$ has for a highest weight vector $w_{(n-1,1)}^{(0)}$.
\end{proof}

{\bf 3. Case $\bm{\beta_1=0}$.} This case is similar to the previous case $\beta_2=0$.

{\bf 4. Case $\bm{\beta_1+\beta_2=0}$.} The algebra $G_d({\mathfrak V}_{\beta})$ satisfies the identity
\[
w_{(2,1)}(y_1,y_2,z_1,z_2)=(y_1-z_1)(y_1z_2-y_2z_1)=0.
\]

\begin{proposition}\label{chi for beta1=beta2}
When $\beta_1+\beta_2=0$ we have
\[
P_n({\mathfrak V}_{\beta})=2M(n)+M(n-1,1),\; n\geq 3.
\]
The two corresponding irreducible $GL_d(K)$-submodules $W_d(n)$ and the $GL_d(K)$-submodule $W_d(n-1,1)$, $n\geq 3$, are generated, respectively, by
$w_{(n)}^{(n-2)}$ and $w_{(n)}^{(n-1)}$ and by $w_{(n-1,1)}^{(n-2)}$.
\end{proposition}

\begin{proof}
The case $n=4$ was handles in Proposition \ref{chi4 for (2,1)}. The explicit form of the highest weight vectors for $\lambda=(4)$ follows immediately
from Lemma \ref{consequences of M(2,1) in G}.
Since $\beta_2=-\beta_1$ the only relation between $w_{(4)}^{(i)}$, $i=1,2,3$, becomes
\[
t^{(4)}_{(1)}=w_{(4)}^{(3)}-2w_{(4)}^{(2)}+w_{(4)}^{(1)}=(y_1-z_1)^2y_1z_1=0
\]
and this implies that the highest weight vectors $w_{(n)}^{(k)}$, $k=1,2,\ldots,n-3$,
can be expressed as linear combinations of $w_{(n)}^{(n-1)}$ and $w_{(n)}^{(n-2)}$.
The consequences of higher degree of the form $w_{(n)}(y_1,z_1)$ are obtained by replacement of some of the the variables in the linearization of $w_{(2,1)}$
by $y_1^pz_1^q$ and multiplication by $y_1$ from the left and by $z_1$ from the right.
It turns out that as elements in $G_d({\mathfrak V}_{\beta}$ in all the cases the consequences are divisible by $y_1z_1(y_1-z_1)^2$. This means that the consequences follow from $t^{(4)}_{(1)}=0$
by multiplication from both sizes and $M(n)$ participates in $P_n({\mathfrak V}_{\beta})$ with multiplicity 2.

The identity $w_{(2,1)}(y_1,y_2,z_1,z_2)=0$ gives that $y_1(y_1z_2-y_2z_1)=(y_1z_2-y_2z_1)z_1$ and this implies that
$w_{(n-1,1)}^{(k+1)}=w_{(n-1,1)}^{(k)}$ for $n\geq 3$ and $k=0,1,\ldots,n-3$.
To complete the proof we have to show that $w_{(n-1,1)}^{(n-2)}$ is not an identity. Direct computation gives that
\[
w_{(n-1,1)}^{(n-2)}(y_1,y_1z_1,z_1,y_1z_1)=y_1^{n-1}z_1(y_1-z_1)=w_{(n+1)}^{(n)}-w_{(n+1)}^{(n-1)}.
\]
If $w_{(n-1,1)}^{(n-2)}=0$ in $G_d({\mathfrak V}_{\beta})$ we would obtain that $w_{(n+1)}^{(n)}=w_{(n+1)}^{(n-1)}$ in $G_d({\mathfrak V}_{\beta})$
which contradicts with the linear independence of these two polynomials. Hence $w_{(n-1,1)}^{(n-2)}\not=0$ in $G_d({\mathfrak V}_{\beta})$.
\end{proof}

{\bf 5. Case $\bm{\beta_1-\beta_2=0}$.} In the language of $G_d({\mathfrak V}_{\beta})$ the variety ${\mathfrak V}_{\beta}$ satisfies the identity
\[
w_{(2,1)}(y_1,y_2,z_1,z_2)=(y_1+z_1)(y_1z_2-y_2z_1)=0.
\]

\begin{proposition}\label{chi for beta1=beta2}
When $\beta_1-\beta_2=0$ we have
\[
P_4({\mathfrak V}_{\beta})=2M(4),\;
P_n({\mathfrak V}_{\beta})=M(n),\; n\geq 5.
\]
The two corresponding irreducible $GL_d(K)$-submodules $W_d(4)$ and the $GL_d(K)$-submodules $W_d(n)$, $n\geq 5$, are generated, respectively, by
$w_{(4)}^{(1)}$ and $w_{(4)}^{(2)}$ and by $w_{(n)}^{(n-1)}$.
\end{proposition}

\begin{proof}
The case $n=4$ was handles in Proposition \ref{chi4 for (2,1)}. The explicit form of the highest weight vectors for $\lambda=(4)$ follows immediately
from Lemma \ref{consequences of M(2,1) in G}.
Since $\beta_1=\beta_2$ the only relation between $w_{(4)}^{(i)}$, $i=1,2,3$, becomes
\[
t^{(4)}_{(1)}=w_{(4)}^{(3)}-w_{(4)}^{(1)}=0.
\]
Hence the highest weight vectors $w_{(4)}^{(3)}$ and $w_{(4)}^{(2)}$ are linearly independent and this closes the case $\lambda=(4)$.

Since $w_{(3,1)}^{(k)}=0$, $k=0,1,2$, we obtain
\[
w_{(3,1)}^{(k)}(y_1,y_1z_1,z_1,y_1z_1)=w_{(5)}^{(k+2)}-w_{(5)}^{(k+1)}
\]
i.e. $w_{(5)}^{(4)}=w_{(5)}^{(3)}=w_{(5)}^{(2)}=w_{(5)}^{(1)}$. But $w_{(5)}^{(4)}\not=0$ in $G_d({\mathfrak V}_{\beta})$ because
it does not vanish in the one-dimensional algebra $A$ which we already considered.

Since $w_{(3,1)}^{(k)}=0$, $k=0,1,2$, by multiplication with $y_1$ and $z_1$ we derive that $w_{(n-1,1)}^{(i)}=0$ for $n\geq 5$ and all $i$.
\end{proof}

\subsection{The lattice of subvarieties}
We follow the same way of considerations as in the case of ${\mathfrak U}_{\alpha}$.

{\bf 1. Case $\bm{\beta_1\beta_2(\beta_1+\beta_2)(\beta_1-\beta_2)\not=0}$.}

\begin{theorem}\label{lattice (2,1) case 1}
When $\beta_1\beta_2(\beta_1+\beta_2)(\beta_1-\beta_2)\not=0$ the consequences in $G_d({\mathfrak V}_{\beta})$ of degree $n+1$ of the polynomial $w_{\lambda}^{(n)}$, $\lambda\vdash n$,
are equivalent to the following polynomials:

{\rm (i)} For $w_{(1)}=x_1$: $w_{(2)}^{(1)}=y_1z_1$, $w_{(1^2)}^{(0)}=y_1z_2-y_2z_1$;

{\rm (ii)} For $w_{(2)}^{(1)}=y_1z_1$: $w_{(3)}^{(2)}=y_1^2z_1$, $w_{(3)}^{(1)}=y_1z_1^2$, $w_{(2,1)}^{(1)}$;

{\rm (iii)} For $w_{(1^2)}^{(0)}=y_1z_2-y_2z_1$: $w_{(3)}^{(2)}-w_{(3)}^{(1)}=y_1z_1(y_1-z_1)$, $w_{(2,1)}^{(1)}$.

{\rm (iv)} For $w_{(3)}=\gamma_1w_{(3)}^{(2)}+\gamma_2w_{(3)}^{(1)}$, $\gamma_1+\gamma_2\not=0$: $w_{(4)}^{(3)},w_{(4)}^{(2)}$, $w_{(3,1)}^{(2)}$.

{\rm (v)} For $w_{(3)}=w_{(3)}^{(2)}-w_{(3)}^{(1)}$: $w_{(4)}^{(3)}-w_{(4)}^{(2)}$.

{\rm (vi)} For $w_{(2,1)}^{(1)}$: $w_{(4)}^{(3)}-w_{(4)}^{(2)}$.

{\rm (vii)} For $w_{(4)}=\gamma_1w_{(4)}^{(2)}+\gamma_2w_{(4)}^{(1)}$, $\gamma_1+\gamma_2\not=0$: $w_{(5)}^{(4)}$.

{\rm (viii)} For $w_{(4)}=w_{(4)}^{(2)}-w_{(4)}^{(1)}$: There are no consequences of degree $5$.

{\rm (ix)} For $w_{(n)}^{(n-1)}$, $n\geq 5$: $w_{(n+1)}^{(n)}$.
\end{theorem}
\vskip-1.5truecm
\[
\begin{picture}(80,160)
\put(0,0){\circle*{3}} \put(0,20){\circle*{3}}
\put(0,40){\circle*{3}} \put(0,60){\circle*{3}}
\put(0,80){\circle*{3}} \put(0,100){\circle*{3}}
\put(0,105){\circle*{1}} \put(0,110){\circle*{1}}
\put(0,115){\circle*{1}}

\put(-20,40){\circle*{3}} \put(-20,60){\circle*{3}}
\put(20,40){\circle*{3}} \put(20,60){\circle*{3}}
\put(40,40){\circle*{3}} \put(40,60){\circle*{3}}

\put(0,0){\vector(0,2){20}}\put(-12,-3){\tiny{$f_1$}}
\put(0,20){\vector(0,2){20}}\put(-12,17){\tiny{$f_2$}}
\put(0,40){\vector(0,2){20}}\put(-10,34){\tiny{$f_3$}}
\put(0,60){\vector(0,2){20}}\put(-10,62){\tiny{$f_4$}}
\put(0,80){\vector(0,2){20}}\put(-12,77){\tiny{$f_5$}}
\put(-12,97){\tiny{$f_6$}}

\put(60,20){\circle*{3}} \put(60,40){\circle*{3}}

\put(0,0){\vector(3,1){60}}
\put(0,20){\vector(-1,1){20}} \put(0,20){\vector(1,1){20}} \put(0,20){\vector(2,1){40}} \put(0,20){\vector(3,1){60}}

\put(60,20){\vector(-2,1){40}} \put(60,20){\vector(0,1){20}}

\put(60,40){\vector(-2,1){40}}

\put(-20,40){\line(60,0){60}} \put(-20,60){\line(60,0){60}}

\put(-20,40){\vector(0,1){20}} \put(20,40){\vector(0,1){20}} \put(40,40){\vector(0,1){20}}

\put(-20,40){\vector(3,2){30}} \put(0,40){\vector(-1,1){20}} \put(0,40){\vector(1,2){10}}
\put(0,40){\vector(2,1){40}} \put(40,40){\vector(-3,2){30}}

\put(-20,60){\vector(1,1){20}} \put(40,60){\vector(-2,1){40}}

\put(-30,34){\tiny{$j_3$}} \put(-30,62){\tiny{$j_4$}}

\put(65,18){\tiny{$h_2$}}

\put(20,34){\tiny{$i_3$}} \put(20,62){\tiny{$i_4$}}

\put(42,34){\tiny{$g_3$}} \put(42,62){\tiny{$g_4$}}
\put(65,34){\tiny{$h_3$}}

\put(-100,-20){\tiny{$f_n=\omega_{(n)}^{(n-1)}$, $g_n=\omega_{(n)}^{(n-2)}$}, $h_n=\omega_{(n-1,1)}^{(n-2)}$, $i_n=\omega_{(n)}^{(n-1)}-\omega_{(n)}^{(n-2)}$,}

\put(-50,-35){\tiny{$j_n=\gamma_1 \omega_{(n)}^{(n-1)}+\gamma_2 \omega_{(n)}^{(n-2)}$, $\gamma_1+\gamma_2 \neq 0$}}

\put(-10,-50){\tiny{$\lambda=(2,1)$, Case 1}}
\end{picture} \ \ \ \ \ \ \
\]

\vskip2truecm

\begin{proof}
By Proposition \ref{chi4 for (2,1)}
\[
P_4({\mathfrak V}_{\beta})=2M(4)+M(3,1),\;P_n({\mathfrak V}_{\beta})=M(n),\,n>4.
\]
The considerations for the cases (i), (ii), (iii) are similar those for ${\mathfrak U}{\alpha}$.
Since $w_{(n+1)}^{(n)}(y_1,z_1)=y_1w_{(n)}^{(n-1)}(y_1,z_1)$ we obtain that $w_{(n+1)}^{(n)}(y_1,z_1)$ is a consequence of $w_{(n)}^{(n-1)}(y_1,z_1)$ in $G_d({\mathfrak V}_{\beta})$ and this is the case (ix).

Cases (iv) and (v). By Lemma \ref{consequences of M(2,1) in G} we have the relation
\[
t^{(4)}_{(1)}=\beta_1w_{(4)}^{(3)}-(\beta_1-\beta_2)w_{(4)}^{(2)}-\beta_2w_{(4)}^{(1)}=0.
\]
The consequences of the form $w_{(4)}$ of $w_{(3)}=\gamma_1w_{(3)}^{(2)}+\gamma_2w_{(3)}^{(1)}=0$ by multiplications from both sides and substitutions are
\[
\gamma_1w_{(4)}^{(3)}+\gamma_2w_{(4)}^{(2)}=0,\;\gamma_1w_{(4)}^{(2)}+\gamma_2w_{(4)}^{(1)}=0,\;
\gamma_1w_{(4)}^{(3)}+2(\gamma_1+\gamma_2)w_{(4)}^{(2)}\gamma_2w_{(4)}^{(1)}=0.
\]
The matrix of the linear system of 4 equations is
\[
\left(\begin{matrix}
\beta_1&-(\beta_1-\beta_2)&-\beta_2\\
\gamma_1&\gamma_2&0\\
0&\gamma_1&\gamma_2\\
\gamma_1&2(\gamma_1+\gamma_2)&\gamma_2
\end{matrix}\right).
\]
The rank of the matrix is equal to 3 if $\gamma_1+\gamma_2\not=0$ and to 2 if $\gamma_1+\gamma_2=0$.
In the former case both $w_{(4)}^{(3)}=0$ and $w_{(4)}^{(2)}=0$ are consequences of $w_{(3)}=0$.
In the latter case the only consequence is $w_{(4)}^{(3)}-w_{(4)}^{(2)}=0$.

The cases (vii) and (viii) are similar to the cases (iv) and (v).

The case (vi) follows immediately because the only consequence $w_{(4)}=0$ of $w_{(2,1)}^{(1)}=0$ is
\[
w_{(2,1)}^{(1)}(y_1,y_1z_1,z_1)=y_1^2z_1(y_1-z_1)=w_{(4)}^{(3)}-w_{(4)}^{(2)}=0.
\]
\end{proof}

{\bf 2. Case $\bm{\beta_2=0}$.} Then $w_{(2,1)}^{(1)}=y_1(y_1z_2-y_2z_1)=0$ in $G_d({\mathfrak V}_{\beta})$.

\begin{theorem}\label{lattice (2,1) case 2}
When $\beta_2=0$ the consequences in $G_d({\mathfrak V}_{\beta})$ of degree $n+1$ of the polynomial $w_{\lambda}^{(n)}$, $\lambda\vdash n$,
are equivalent to the following polynomials:

{\rm (i)} For $w_{(1)}=x_1$: $w_{(2)}^{(1)}=y_1z_1$, $w_{(1^2)}^{(0)}=y_1z_2-y_2z_1$;

{\rm (ii)} For $w_{(2)}^{(1)}=y_1z_1$: $w_{(3)}^{(2)}=y_1^2z_1$, $w_{(3)}^{(1)}=y_1z_1^2$, $w_{(2,1)}^{(0)}$;

{\rm (iii)} For $w_{(1^2)}^{(0)}=y_1z_2-y_2z_1$: $w_{(3)}^{(2)}-w_{(3)}^{(1)}=y_1z_1(y_1-z_1)$, $w_{(2,1)}^{(0)}$.

{\rm (iv)} For $w_{(n)}=\gamma_1w_{(n)}^{(2)}+\gamma_2w_{(n)}^{(1)}$, $n\geq 3$, $(\gamma_1+\gamma_2)\gamma_2\not=0$: $w_{(n+1)}^{(2)},w_{(n+1)}^{(1)}$, $w_{(n,1)}^{(0)}$.

{\rm (v)} For $w_{(n)}=w_{(n)}^{(2)}-w_{(n)}^{(1)}$, $n\geq 3$: $w_{(n+1)}^{(2)}-w_{(n+1)}^{(1)}$, $w_{(n,1)}^{(0)}$.

{\rm (vi)} For $w_{(n)}=w_{(n)}^{(2)}$, $n\geq 3$: $w_{(n+1)}^{(2)}$.

{\rm (vii)} For $w_{(n-1,1)}^{(0)}$, $n\geq 3$: $w_{(n+1)}^{(2)}-w_{(n+1)}^{(1)}$, $w_{(n,1)}^{(0)}$.
\end{theorem}
\vskip-2truecm
\[
\begin{picture}(80,160)
\put(0,0){\circle*{3}} \put(0,20){\circle*{3}}
\put(0,40){\circle*{3}} \put(0,60){\circle*{3}}
\put(0,80){\circle*{3}}

\put(-20,95){\circle*{1}} \put(-20,100){\circle*{1}}
\put(-20,105){\circle*{1}}

\put(0,95){\circle*{1}} \put(0,100){\circle*{1}}
\put(0,105){\circle*{1}}

\put(20,95){\circle*{1}} \put(20,100){\circle*{1}}
\put(20,105){\circle*{1}}

\put(40,95){\circle*{1}} \put(40,100){\circle*{1}}
\put(40,105){\circle*{1}}

\put(60,95){\circle*{1}} \put(60,100){\circle*{1}}
\put(60,105){\circle*{1}}

\put(-20,40){\circle*{3}} \put(-20,60){\circle*{3}} \put(-20,80){\circle*{3}}
\put(20,40){\circle*{3}} \put(20,60){\circle*{3}} \put(20,80){\circle*{3}}
\put(40,40){\circle*{3}} \put(40,60){\circle*{3}} \put(40,80){\circle*{3}}

\put(0,0){\vector(0,2){20}}\put(-12,-3){\tiny{$f_1$}}
\put(0,20){\vector(0,2){20}}\put(-12,17){\tiny{$f_2$}}
\put(0,40){\vector(0,2){20}}\put(-10,34){\tiny{$f_3$}}
\put(0,60){\vector(0,2){20}}\put(-10,62){\tiny{$f_4$}}
\put(0,80){\line(0,1){10}}\put(-10,82){\tiny{$f_5$}}
\put(-20,80){\line(0,1){10}} \put(20,80){\line(0,1){10}} \put(40,80){\line(0,1){10}} \put(60,80){\line(0,1){10}}

\put(60,20){\circle*{3}} \put(60,40){\circle*{3}}
\put(60,60){\circle*{3}} \put(60,80){\circle*{3}}

\put(0,0){\vector(3,1){60}}
\put(0,20){\vector(-1,1){20}} \put(0,20){\vector(1,1){20}} \put(0,20){\vector(2,1){40}} \put(0,20){\vector(3,1){60}}

\put(60,20){\vector(-2,1){40}} \put(60,20){\vector(0,1){20}} \put(60,40){\vector(0,1){20}} \put(60,60){\vector(0,1){20}} \put(60,60){\vector(-2,1){40}}

\put(60,40){\vector(-2,1){40}}

\put(-20,40){\line(60,0){60}} \put(-20,60){\line(60,0){60}} \put(-20,80){\line(60,0){60}}

\put(-20,40){\vector(0,1){20}} \put(-20,60){\vector(0,1){20}} \put(20,40){\vector(0,1){20}} \put(20,60){\vector(0,1){20}} \put(40,40){\vector(0,1){20}} \put(40,60){\vector(0,1){20}}

\put(-20,40){\vector(3,1){60}} \put(-20,40){\vector(4,1){80}}

\put(20,40){\vector(2,1){40}}
\put(40,40){\vector(1,1){20}} \put(40,40){\vector(-3,1){60}} \put(40,40){\vector(-2,1){40}}

\put(20,60){\vector(2,1){40}}
\put(40,60){\vector(1,1){20}} \put(40,60){\vector(-3,1){60}} \put(40,60){\vector(-2,1){40}}

\put(-20,60){\vector(3,1){60}} \put(-20,60){\vector(4,1){80}}

\put(-30,34){\tiny{$j_3$}} \put(-30,62){\tiny{$j_4$}} \put(-30,82){\tiny{$j_5$}}

\put(65,18){\tiny{$h_2$}} \put(65,34){\tiny{$h_3$}} \put(65,62){\tiny{$h_4$}} \put(65,82){\tiny{$h_5$}}

\put(10,34){\tiny{$i_3$}} \put(10,62){\tiny{$i_4$}} \put(10,82){\tiny{$i_5$}}

\put(42,34){\tiny{$g_3$}} \put(42,62){\tiny{$g_4$}} \put(42,82){\tiny{$g_5$}}

\put(-110,-20){\tiny{$f_2=\omega_{(2)}^{(1)}$, $f_n=\omega_{(n)}^{(2)}$, $n\geq 3$, $g_n=\omega_{(n)}^{(1)}$}, $h_n=\omega_{(n-1,1)}^{(0)}$, $i_n=\omega_{(n)}^{(2)}-\omega_{(n)}^{(1)}$,}

\put(-50,-35){\tiny{$j_n=\gamma_1 \omega_{(n)}^{(2)}+\gamma_2 \omega_{(n)}^{(1)}$, $(\gamma_1+\gamma_2)\gamma_2 \neq 0$}}

\put(-10,-50){\tiny{$\lambda=(2,1)$, Case 2}}
\end{picture}
\]

\vskip2truecm

\begin{proof}
The cases (i), (ii) and (iii) are as in the previous considerations, (vii) is also clear because
\[
w_{(n-1,1)}^{(0)}(y_1,y_1z_1,z_1,y_1z_1)=w_{(n+1)}^{(2)}-w_{(n+1)}^{(1)},\;w_{(n,1)}^{(0)}=w_{(n-1,1)}^{(0)}z_1.
\]

(iv) The consequences $w_{(n+1)}=0$ of $w_{(n)}=\gamma_1w_{(n)}^{(2)}+\gamma_2w_{(n)}^{(1)}=0$ are
\[
y_1w_{(n)}=y_1(\gamma_1w_{(n)}^{(2)}+\gamma_2w_{(n)}^{(1)})=\gamma_1w_{(n+1)}^{(3)}+\gamma_2w_{(n+1)}^{(2)}=0,
\]
\[
w_{(n)}z_1=(\gamma_1w_{(n)}^{(2)}+\gamma_2w_{(n)}^{(1)})z_1=\gamma_1w_{(n+1)}^{(2)}+\gamma_2w_{(n+1)}^{(1)}=0,
\]
\[
\Delta_1(y_1z_1,w_{(n)})=(n-2)\gamma_1w_{(n+1)}^{(3)}+(2\gamma_1+(n-1)\gamma_2)w_{(n+1)}^{(2)}+\gamma_2w_{(n+1)}^{(1)}=0.
\]
Together with the equality
\[
\beta_1w_{(n+1)}^{(3)}-(\beta_1-\beta_2)w_{(n+1)}^{(2)}-\beta_2w_{(n+1)}^{(1)}=0
\]
which due to $\beta_2=0$ becomes $w_{(n+1)}^{(3)}=w_{(n+1)}^{(2)}$, we obtain
a linear homogeneous system with a matrix
\begin{equation}\label{matrix for (2,1)}
\left(\begin{matrix}\gamma_1&\gamma_2&0\\
0&\gamma_1&\gamma_2\\
(n-2)\gamma_1&2\gamma_1+(n-1)\gamma_2&\gamma_2\\
1&-1&0
\end{matrix}\right).
\end{equation}
The system has the trivial solution only, i.e.
$w_{(n+1)}^{(2)}=w_{(n+1)}^{(1)}=0$,
and both $w_{(n+1)}^{(2)}$ and $w_{(n+1)}^{(1)}$ are consequences of $w_{(n)}$.

Now $0=nw_{(n)}(y_1,z_1)z_2-\Delta_{12}(y_1,z_1)z_1=2(\gamma_1y_1+\gamma_2z_1)z_1^{n-2}(y_1z_2-y_2z_1)$
\[
=2(\gamma_1w_{(n,1)}^{(1)}+\gamma_2w_{(n,1)}^{(0)}=\gamma_2w_{(n,1)}^{(0)}
\]
and $w_{(n,1)}^{(0)}=0$ because $\gamma_2\not=0$.

(v) Since $\gamma_2=-\gamma_1$ the only nonzero solution of the system with matrix (\ref{matrix for (2,1)}) is
$w_{(n+1)}^{(2)}=w_{(n+1)}^{(1)}$. As in the case (iv), $\gamma_2\not=0$ and we have that $w_{(n,1)}^{(0)}=0$.

(vi) Since $\gamma_2=0$ the system with matrix (\ref{matrix for (2,1)}) has the solution $w_{(n+1)}^{(2)}=0$.
The degree in $y_1,y_2$ in both equalities
\[
w_{(2,1)}^{(1)}=y_1(y_1z_2-y_2z_1)=0,\;w_{(n)}=w_{(n)}^{(2)}=y_1^2z_1^{n-2}=0
\]
is equal to 2 and hence $\deg_y(w)\geq 2$ for all their consequences $w=0$. Since $\deg_y(w_{(n,1)}^{(0)})=1$, it does not follow from
$w_{(2,1)}^{(1)}=w_{(n)}^{(2)}=0$.
\end{proof}

{\bf 3. Case $\bm{\beta_1=0}$.} This case is similar to the previous case $\beta_2=0$.

{\bf 4. Case $\bm{\beta_1+\beta_2=0}$.}
The variety of bicommutative algebras ${\mathfrak V}_{\beta}$ is defined by the identity
$x_1[x_1,x_2]=[x_1,x_2]x_1$ which in $G_d({\mathfrak V}_{\beta})$ becomes
\[
w_{(2,1)}(y_1,y_2,z_1,z_2)=w_{(2,1)}^{(1)}-w_{(2,1)}^{(0)}=(y_1-z_1)(y_1z_2-y_2z_1)=0.
\]
\begin{theorem}\label{lattice (2,1) case 4}
When $\beta_1+\beta_2=0$ the consequences in $G_d({\mathfrak V}_{\beta})$ of degree $n+1$ of the polynomial $w_{\lambda}^{(n)}$, $\lambda\vdash n$,
are equivalent to the following polynomials:

{\rm (i)} For $w_{(1)}=x_1$: $w_{(2)}^{(1)}=y_1z_1$, $w_{(1^2)}^{(0)}=y_1z_2-y_2z_1$;

{\rm (ii)} For $w_{(2)}^{(1)}=y_1z_1$: $w_{(3)}^{(2)}=y_1^2z_1$, $w_{(3)}^{(1)}=y_1z_1^2$, $w_{(2,1)}^{(1)}$;

{\rm (iii)} For $w_{(n)}=\gamma_1w_{(n)}^{(n-1)}+\gamma_2w_{(n)}^{(n-2)}$, $n\geq 3$, $(\gamma_1+\gamma_2)\not=0$: $w_{(n+1)}^{(n)},w_{(n+1)}^{(n-1)}$, $w_{(n,1)}^{(n-1)}$.

{\rm (iv)} For $w_{(n)}=w_{(n)}^{(n-1)}-w_{(n)}^{(n-2)}$, $n\geq 3$: $w_{(n+1)}^{(n)}-w_{(n+1)}^{(n-1)}$, $w_{(n,1)}^{(n-1)}$.

{\rm (v)} For $w_{(n-1,1)}^{(n-2)}$, $n\geq 3$: $w_{(n+1)}^{(n-1)}-w_{(n+1)}^{(n-2)}$, $w_{(n,1)}^{(n-1)}$.
\end{theorem}
\vskip-2truecm
\[
\begin{picture}(80,160)
\put(0,0){\circle*{3}} \put(0,20){\circle*{3}}
\put(0,40){\circle*{3}} \put(0,60){\circle*{3}}
\put(0,80){\circle*{3}}

\put(-20,95){\circle*{1}} \put(-20,100){\circle*{1}}
\put(-20,105){\circle*{1}}

\put(0,95){\circle*{1}} \put(0,100){\circle*{1}}
\put(0,105){\circle*{1}}

\put(20,95){\circle*{1}} \put(20,100){\circle*{1}}
\put(20,105){\circle*{1}}

\put(40,95){\circle*{1}} \put(40,100){\circle*{1}}
\put(40,105){\circle*{1}}

\put(60,95){\circle*{1}} \put(60,100){\circle*{1}}
\put(60,105){\circle*{1}}

\put(-20,40){\circle*{3}} \put(-20,60){\circle*{3}} \put(-20,80){\circle*{3}}
\put(20,40){\circle*{3}} \put(20,60){\circle*{3}} \put(20,80){\circle*{3}}
\put(40,40){\circle*{3}} \put(40,60){\circle*{3}} \put(40,80){\circle*{3}}

\put(0,0){\vector(0,2){20}}\put(-12,-3){\tiny{$f_1$}}
\put(0,20){\vector(0,2){20}}\put(-12,17){\tiny{$f_2$}}
\put(0,40){\vector(0,2){20}}\put(-10,34){\tiny{$f_3$}}
\put(0,60){\vector(0,2){20}}\put(-10,62){\tiny{$f_4$}}
\put(0,80){\line(0,1){10}}\put(-10,82){\tiny{$f_5$}}
\put(-20,80){\line(0,1){10}} \put(20,80){\line(0,1){10}} \put(40,80){\line(0,1){10}} \put(60,80){\line(0,1){10}}

\put(60,20){\circle*{3}} \put(60,40){\circle*{3}}
\put(60,60){\circle*{3}} \put(60,80){\circle*{3}}

\put(0,0){\vector(3,1){60}}
\put(0,20){\vector(-1,1){20}} \put(0,20){\vector(1,1){20}} \put(0,20){\vector(2,1){40}} \put(0,20){\vector(3,1){60}}

\put(0,40){\vector(-1,1){20}} \put(0,40){\vector(1,1){20}} \put(0,40){\vector(2,1){40}}

\put(0,60){\vector(-1,1){20}} \put(0,60){\vector(1,1){20}} \put(0,60){\vector(2,1){40}}

\put(60,20){\vector(-2,1){40}} \put(60,20){\vector(0,1){20}} \put(60,40){\vector(0,1){20}} \put(60,60){\vector(0,1){20}}

\put(60,40){\vector(-2,1){40}}

\put(60,60){\vector(-2,1){40}}

\put(-20,40){\line(60,0){60}} \put(-20,60){\line(60,0){60}} \put(-20,80){\line(60,0){60}}

\put(-20,40){\vector(0,1){20}} \put(-20,60){\vector(0,1){20}} \put(20,40){\vector(0,1){20}} \put(20,60){\vector(0,1){20}} \put(40,40){\vector(0,1){20}} \put(40,60){\vector(0,1){20}}

\put(-20,40){\vector(3,1){60}} \put(-20,40){\vector(4,1){80}}

\put(20,40){\vector(2,1){40}}
\put(40,40){\vector(1,1){20}} \put(40,40){\vector(-3,1){60}} \put(40,40){\vector(-2,1){40}}

\put(20,60){\vector(2,1){40}}
\put(40,60){\vector(1,1){20}} \put(40,60){\vector(-3,1){60}} \put(40,60){\vector(-2,1){40}}

\put(-20,60){\vector(3,1){60}} \put(-20,60){\vector(4,1){80}}

\put(-30,34){\tiny{$j_3$}} \put(-30,62){\tiny{$j_4$}} \put(-30,82){\tiny{$j_5$}}

\put(65,18){\tiny{$h_2$}} \put(65,34){\tiny{$h_3$}} \put(65,62){\tiny{$h_4$}} \put(65,82){\tiny{$h_5$}}

\put(10,34){\tiny{$i_3$}} \put(10,62){\tiny{$i_4$}} \put(10,82){\tiny{$i_5$}}

\put(42,34){\tiny{$g_3$}} \put(42,62){\tiny{$g_4$}} \put(42,82){\tiny{$g_5$}}

\put(-110,-20){\tiny{$f_2=\omega_{(2)}^{(1)}$, $f_n=\omega_{(n)}^{(n-1)}$, $n\geq 3$, $g_n=\omega_{(n)}^{(n-2)}$}, $h_n=\omega_{(n-1,1)}^{(n-2)}$, $i_n=\omega_{(n)}^{(n-1)}-\omega_{(n)}^{(n-2)}$,}

\put(-50,-35){\tiny{$j_n=\gamma_1 \omega_{(n)}^{(n-1)}+\gamma_2 \omega_{(n)}^{(n-2)}$, $(\gamma_1+\gamma_2)\gamma_2 \neq 0$}}

\put(-10,-50){\tiny{$\lambda=(2,1)$, Case 4}}
\end{picture}
\]
\vskip2truecm

\begin{proof}
The cases (i) and (ii) are the same as in the cases $\beta_1\beta_2(\beta_1+\beta_2)(\beta_1-\beta_2)\not=0$ and $\beta_2=0$.
The identity $t^{(4)}_{(1)}=0$ from Lemma \ref{consequences of M(2,1) in G} becomes
$w_{(4)}^{(3)}-2w_{(4)}^{(2)}+w_{(4)}^{(1)}=0$.
Now the case (v) is clear:
\[
y_1w_{(n-1,1)}^{(n-2)}=w_{(n,1)}^{(n-1)}=0,\;w_{(n-1,1)}^{(n-2)}(y_1,y_1z_1,z_1,y_1z_1)=w_{(n+1)}^{(n-1)}-w_{(n+1)}^{(n-2)}=0
\]
and the latter together with the consequence $w_{(n+1)}^{(n)}-2w_{(n+1)}^{(n-1)}+w_{(n+1)}^{(n-2)}=0$ of $w_{(n+1)}^{(n)}-2w_{(n+1)}^{(n-1)}+w_{(n+1)}^{(n-2)}=0$
implies $w_{(n+1)}^{(n)}-w_{(n+1)}^{(n-1)}=0$.
The one dimensional algebra $A$ which we already used belongs to ${\mathfrak V}_{\beta}$, satisfies the identity $w_{(n-1,1)}^{(n-2)}=0$
and $\rho_1w_{(n+1)}^{(n-1)}+\rho_2w_{(n+1)}^{(n-2)}\not=0$ in $A$ if $\rho_1+\rho_2\not=0$.

(iii) Let $w_{(n)}=\gamma_1w_{(n)}^{(n-1)}+\gamma_2w_{(n)}^{(n-2)}=0$, $n\geq 3$. As in the case $\beta_2=0$, the consequences $w_{(n+1)}$ are linear combinations of
\[
\gamma_1w_{(n+1)}^{(n)}+\gamma_2w_{(n+1)}^{(n-1)}=0,
\]
\[
\gamma_1w_{(n+1)}^{(n-1)}+\gamma_2w_{(n+1)}^{(n-2)}=0,
\]
\[
\gamma_1((n-1)w_{(n+1)}^{(n-1)}+w_{(n+1)}^{(n)})+\gamma_2((n-2)w_{(n+1)}^{(n-2)}+2w_{(n+1)}^{(n-1)})=0
\]
\[
w_{(n+1)}^{(n)}-2w_{(n+1)}^{(n-1)}+w_{(n+1)}^{(n-2)}=0.
\]
When $\gamma_1+\gamma_2\not=0$ the system has only the trivial solution $w_{(n+1)}^{(n)}=w_{(n+1)}^{(n-1)}=0$.

Direct calculations show that for
\[
w_{(n)}=\gamma_1w_{(n)}^{(n-1)}+\gamma_2w_{(n)}^{(n-2)}
\]
\[
\Delta_1(y_1z_2-y_2z_1,w_{(n)}(y_1,z_1))=\gamma_1w_{(n,1)}^{(n-1)}+((n-1)\gamma_1+2\gamma_2)w_{(n,1)}^{(n-2)}+(n-2)\gamma_2w_{(n,1)}^{(n-3)}=0
\]
and since $w_{(n,1)}^{(i)}=w_{(n,1)}^{(n-1)}$, $i=0,1,\ldots,n-2$, we obtain $n(\gamma_1+\gamma_2)w_{(n,1)}^{(n-1)}=0$ and hence $w_{(n,1)}^{(n-1)}=0$.

(iv) We obtain immediately $y_1(w_{(n)}^{(n-1)}-w_{(n)}^{(n-2)})=w_{(n+1)}^{(2)}-w_{(n+1)}^{(1)}=0$ and
$\rho_1w_{(n+1)}^{(n-1)}+\rho_2w_{(n+1)}^{(n-2)}\not=0$ does not follow $w_{(n)}=0$ because does not vanish on the algebra $A$ if $\rho_1+\rho_2\not=0$.

We have
\[
0=n(w_{(n)}^{(n-1)}(y_1,z_1)-w_{(n)}^{(n-2)}(y_1,z_1))z_1-\Delta_{12}(w_{(n)}^{(n-1)}(y_1,z_1)-w_{(n)}^{(n-2)(y_1,z_1)})z_1
\]
\[
=(n-1)w_{(n,1)}^{(n-2)}-(n-2)w_{(n,1)}^{(n-3)}.
\]
Since $w_{(n,1)}^{(n-3)}=w_{(n,1)}^{(n-2)}=w_{(n,1)}^{(n-1)}$ in $G_d({\mathfrak V}_{\beta})$, we obtain that $w_{(n,1)}^{(n-1)}=0$.
\end{proof}

{\bf 5. Case $\bm{\beta_1-\beta_2=0}$.} The algebra $G_d({\mathfrak V}_{\beta})$ satisfies the identity
\[
w_{(2,1)}=w_{(2,1)}^{(1)}+w_{(2,1)}^{(0)}=(y_1+z_1)(y_1z_2-y_2z_1)=0
\]
and it follows from Proposition \ref{chi for beta1=beta2} that
\[
\chi_n({\mathfrak V}_{\beta})=\begin{cases}
\chi_{(1)},\text{ if }n=1,\\
\chi_{(2)}+\chi_{(1^2)},\text{ if }n=2,\\
2\chi_{(3)}+\chi_{(2,1)},\text{ if }n=3,\\
2\chi_{(4)},\text{ if }n=4,\\
\chi_{(n)},\text{ if }n\geq 5.\\
\end{cases}
\]
Also, in $G_d({\mathfrak V}_{\beta})$ we have
\[
w_{(2,1)}(y_1,y_1z_1,z_1,y_1z_1)=(y_1^2-z_1^2)y_1z_1=w_{(4)}^{(3)}-w_{(4)}^{(1)}=0.
\]
From the proof of Proposition \ref{chi for beta1=beta2} for degrees $n\geq 5$ we have $w_{(n)}^{(i)}-w_{(n)}^{(i-1)}=0$, $i=2,\ldots,n-1$.

\begin{theorem}\label{lattice (2,1) case 5}
When $\beta_1-\beta_2=0$ the consequences in $G_d({\mathfrak V}_{\beta})$ of degree $n+1$ of the polynomial $w_{\lambda}^{(n)}$, $\lambda\vdash n$,
are equivalent to the following polynomials:

{\rm (i)} For $w_{(1)}=x_1$: $w_{(2)}^{(1)}=y_1z_1$, $w_{(1^2)}^{(0)}=y_1z_2-y_2z_1$;

{\rm (ii)} For $w_{(2)}^{(1)}=y_1z_1$: $w_{(3)}^{(2)}=y_1^2z_1$, $w_{(3)}^{(1)}=y_1z_1^2$, $w_{(2,1)}^{(1)}$;

{\rm (iii)} For $w_{(1^2)}^{(0)}=y_1z_2-y_2z_1$: $w_{(3)}^{(2)}-w_{(3)}^{(1)}=y_1z_1(y_1-z_1)$, $w_{(2,1)}^{(1)}$.

{\rm (iv)} For $w_{(3)}=\gamma_1w_{(3)}^{(2)}+\gamma_2w_{(3)}^{(1)}$, $(\gamma_1+\gamma_2)\not=0$: $w_{(4)}^{(3)},w_{(4)}^{(2)}$.

{\rm (v)} For $w_{(3)}=w_{(3)}^{(2)}-w_{(3)}^{(1)}$: $w_{(4)}^{(3)}-w_{(4)}^{(2)}$.

{\rm (vi)} For $w_{(2,1)}=w_{(2,1)}^{(1)}$: $w_{(4)}^{(3)}-w_{(4)}^{(2)}$.

{\rm (vii)} For $w_{(4)}=\gamma_1w_{(4)}^{(3)}+\gamma_2w_{(4)}^{(2)}$, $(\gamma_1+\gamma_2)\not=0$: $w_{(5)}^{(4)}$.

{\rm (viii)} For $w_{(4)}=w_{(4)}^{(3)}-w_{(4)}^{(2)}$: There are no consequences of degree $5$.

{\rm (ix)} For $w_{(n)}^{(n-1)}$, $n\geq 5$: $w_{(n+1)}^{(n)}$.
\end{theorem}
\vfill\eject
\phantom{xxx}
\vskip1truecm
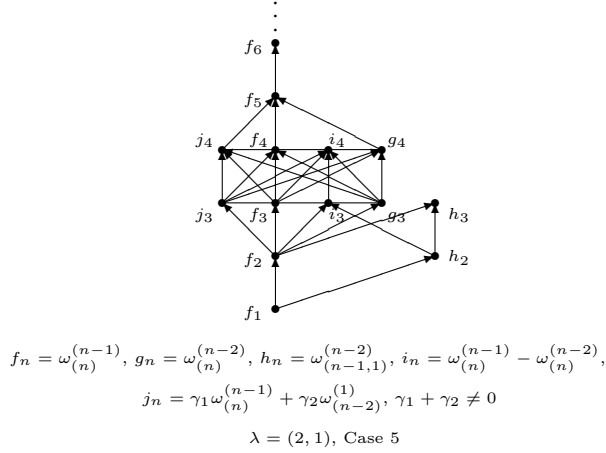
\[
\begin{picture}(80,80)
\put(0,0){\circle*{3}} \put(0,20){\circle*{3}}
\put(0,40){\circle*{3}} \put(0,60){\circle*{3}}
\put(0,80){\circle*{3}} \put(0,100){\circle*{3}}
\put(0,105){\circle*{1}} \put(0,110){\circle*{1}}
\put(0,115){\circle*{1}}

\put(-20,40){\circle*{3}} \put(-20,60){\circle*{3}}
\put(20,40){\circle*{3}} \put(20,60){\circle*{3}}
\put(40,40){\circle*{3}} \put(40,60){\circle*{3}}

\put(0,0){\vector(0,2){20}}\put(-12,-3){\tiny{$f_1$}}
\put(0,20){\vector(0,2){20}}\put(-12,17){\tiny{$f_2$}}
\put(0,40){\vector(0,2){20}}\put(-10,34){\tiny{$f_3$}}
\put(0,60){\vector(0,2){20}}\put(-10,62){\tiny{$f_4$}}
\put(0,80){\vector(0,2){20}}\put(-12,77){\tiny{$f_5$}}
\put(-12,97){\tiny{$f_6$}}

\put(60,20){\circle*{3}} \put(60,40){\circle*{3}}

\put(0,0){\vector(3,1){60}}
\put(0,20){\vector(-1,1){20}} \put(0,20){\vector(1,1){20}} \put(0,20){\vector(2,1){40}} \put(0,20){\vector(3,1){60}}

\put(60,20){\vector(-2,1){40}} \put(60,20){\vector(0,1){20}}

\put(-20,40){\line(60,0){60}} \put(-20,60){\line(60,0){60}}

\put(-20,40){\vector(0,1){20}} \put(20,40){\vector(0,1){20}} \put(40,40){\vector(0,1){20}} \put(40,40){\vector(-2,1){40}}
\put(40,40){\vector(-3,1){60}}

\put(-20,40){\vector(1,1){20}} \put(-20,40){\vector(2,1){40}} \put(-20,40){\vector(3,1){60}}

\put(0,40){\vector(-1,1){20}} \put(0,40){\vector(1,1){20}}
\put(0,40){\vector(2,1){40}}
\put(40,40){\vector(-1,1){20}}

\put(-20,60){\vector(1,1){20}} \put(40,60){\vector(-2,1){40}}

\put(-30,34){\tiny{$j_3$}} \put(-30,62){\tiny{$j_4$}}

\put(65,18){\tiny{$h_2$}}

\put(20,34){\tiny{$i_3$}} \put(20,62){\tiny{$i_4$}}

\put(42,34){\tiny{$g_3$}} \put(42,62){\tiny{$g_4$}}
\put(65,34){\tiny{$h_3$}}

\put(-100,-20){\tiny{$f_n=\omega_{(n)}^{(n-1)}$, $g_n=\omega_{(n)}^{(n-2)}$}, $h_n=\omega_{(n-1,1)}^{(n-2)}$, $i_n=\omega_{(n)}^{(n-1)}-\omega_{(n)}^{(n-2)}$,}

\put(-50,-35){\tiny{$j_n=\gamma_1 \omega_{(n)}^{(n-1)}+\gamma_2 \omega_{(n-2)}^{(1)}$, $\gamma_1+\gamma_2 \neq 0$}}

\put(-10,-50){\tiny{$\lambda=(2,1)$, Case 5}}
\end{picture}
\]
\vskip2truecm

\begin{proof}
The cases (i), (ii) and (iii) are as in the case $\beta_1\beta_2(\beta_1+\beta_2)(\beta_1-\beta_2)\not=0$.
The cases (vi) and (ix) is also clear.

(iv) and (v) The consequences of $w_{(3)}=0$ are
\[
y_1w(3)=\gamma_1w_{(4)}^{(3)}+\gamma_2w_{(4)}^{(2)}=0,\;w(3)z_1=\gamma_1w_{(4)}^{(2)}+\gamma_2w_{(4)}^{(1)}=0,
\]
\[
\Delta_1(y_1z_1,w(3))=\gamma_1w_{(4)}^{(3)}+2(\gamma_1+\gamma_2)w_{(4)}^{(2)}+\gamma_2w_{(4)}^{(1)}=0.
\]
Together with the identity $w_{(4)}^{(3)}-w_{(4)}^{(1)}=0$ we obtain a system with matrix
\[
\left(\begin{matrix}1&0&-1\\
\gamma_1&\gamma_2&0\\
0&\gamma_1&\gamma_2\\
\gamma_1&2(\gamma_1+\gamma_2)&\gamma_2\end{matrix}\right).
\]
The rank of the matrix is equal to 3 if $\gamma_1+\gamma_2\not=0$ and hence $w_{(4)}^{(3)}=w_{(4)}^{(1)}=0$.
When $\gamma_1+\gamma_2=0$ the rank is equal to 2 and the system has the only solution $w_{(4)}^{(3)}-w_{(4)}^{(1)}=0$.

(vii) and (viii) We have $y_1w_{(4)}=\gamma_1w_{(5)}^{(3)}+\gamma_2w_{(5)}^{(2)}=0$. Since $w_{(5)}^{(4)}=w_{(5)}^{(3)}=w_{(5)}^{(2)}$
if $\gamma_1+\gamma_2\not=0$ we derive $w_{(5)}^{(4)}=w_{(5)}^{(3)}=w_{(5)}^{(2)}=0$. If $\gamma_1+\gamma_2=0$, the polynomial identity $w_{(5)}^{(4)}=0$
does not follow from $w_{(4)}^{(3)}-w_{(4)}^{(2)}=0$ because $w_{(5)}^{(4)}$ is not zero in the one-dimensional algebra $A$ and
$w_{(4)}^{(3)}-w_{(4)}^{(2)}=0$ is an identity there.
\end{proof}

\section{Varieties with distributive lattice of subvarieties}
As we mentioned in the introduction the lattice $L({\mathfrak W})$ of the subvarieties of the variety $\mathfrak W$ is distributive if and only if
all multiplicities $m(\lambda)$ in the cocharacter sequence $\chi_n({\mathfrak W})$ satisfy $m(\lambda)\leq 1$.

\begin{theorem}\label{distributive lattice}
A variety $\mathfrak W$ of bicommutative algebras has a distributive lattice of subvarieties if and only it it satisfies the identities
\[
\alpha_1x_1(x_1x_1)+\alpha_2(x_1x_1)x_1=0,\,\beta_1x_1[x_1,x_2]+\beta_2[x_1,x_2]x_1=0
\]
for some $(\alpha_1,\alpha_2),(\beta_1,\beta_2)\not=(0,0)$.
\end{theorem}

\begin{proof}
By the results of Section 3 for the varieties ${\mathfrak U}_{\alpha}$ the only case with $m(\lambda)>1$ is $m(2,1)=2$
and this means that ${\mathfrak U}_{\alpha}$ does not satisfy an identity of the form $\beta_1x_1[x_1,x_2]+\beta_2[x_1,x_2]x_1=0$.
If we add this identity we obtain $m(\lambda)\leq 1$ for all $\lambda\vdash n$, $n=1,2,\ldots$.
\end{proof}

\begin{example}
The results in Section 3 allow to describe the lattices of all varieties of bicommutative algebras with distributive lattice of subvarieties.
We choose the highest weight vector $w_{(2,1)}$ in the form $\gamma_1w_{(2,1)}^{(1)}+\gamma_2w_{(2,1)}^{(0)}$ with the condition
$\beta_1\gamma_2-\beta_2\gamma_1\not=0$ which means that it is not an identity in the variety ${\mathfrak U}_{\alpha}\cap{\mathfrak V}_{\beta}$.
The results are given in the figures below.
\vskip-1truecm
\hskip3truecm
$
\begin{picture}(80,120)
\put(0,0){\circle*{3}} \put(0,20){\circle*{3}}
\put(0,40){\circle*{3}}

\put(20,20){\circle*{3}} \put(30,40){\circle*{3}}

\put(0,0){\vector(0,1){20}}\put(-12,-3){\tiny{$f_1$}}
\put(0,20){\vector(0,1){20}}\put(-12,17){\tiny{$f_2$}}
\put(-12,37){\tiny{$f_3$}}

\put(0,0){\vector(1,1){20}}\put(25,17){\tiny{$g_2$}}
\put(35,37){\tiny{$k_3$}}

\put(0,20){\vector(3,2){30}}

\put(20,20){\vector(1,2){10}} \put(20,20){\vector(-1,1){20}}

\put(-70,-20){\tiny{$f_n=\omega_{(n)}^{(n-1)}$, $g_2=\omega_{(1^2)}$, $k_3=\gamma_1\omega_{(2,1)}^{(1)}+\gamma_2\omega_{(2,1)}^{(0)}$}}

\put(-30,-35){\tiny{$\alpha_1\alpha_2(\alpha_1+\alpha_2)(\alpha_1-\alpha_2)\neq0$}}
\end{picture}
$
\hskip4truecm
$
\begin{picture}(80,120)
\put(0,0){\circle*{3}} \put(0,20){\circle*{3}}
\put(0,40){\circle*{3}}

\put(20,20){\circle*{3}} \put(30,40){\circle*{3}}

\put(0,0){\vector(0,1){20}}\put(-12,-3){\tiny{$f_1$}}
\put(0,20){\vector(0,1){20}}\put(-12,17){\tiny{$f_2$}}
\put(-12,37){\tiny{$f_3$}}

\put(0,0){\vector(1,1){20}}\put(25,17){\tiny{$g_2$}}
\put(35,37){\tiny{$k_3$}}

\put(0,20){\vector(3,2){30}}

\put(20,20){\vector(1,2){10}} \put(20,20){\vector(-1,1){20}}

\put(-70,-20){\tiny{$f_n=\omega_{(n)}^{(1)}$, $g_2=\omega_{(1^2)}$, $k_3=\gamma_1\omega_{(2,1)}^{(0)}+\gamma_2\omega_{(2,1)}^{(1)}$}}

\put(-10,-35){\tiny{$\alpha_2 = 0$, $\beta_2\neq 0$}}
\end{picture}
$

\hskip3truecm
$
\begin{picture}(80,120)
\put(0,0){\circle*{3}} \put(0,20){\circle*{3}}
\put(0,40){\circle*{3}} \put(0,60){\circle*{3}}
\put(0,65){\circle*{1}} \put(0,70){\circle*{1}}
\put(0,75){\circle*{1}}

\put(30,20){\circle*{3}} \put(30,40){\circle*{3}}
\put(30,60){\circle*{3}}
\put(30,65){\circle*{1}} \put(30,70){\circle*{1}}
\put(30,75){\circle*{1}}

\put(0,0){\vector(0,2){20}}\put(-12,-3){\tiny{$f_1$}}
\put(0,20){\vector(0,2){20}}\put(-12,17){\tiny{$f_2$}}
\put(0,40){\vector(0,2){20}}\put(-12,37){\tiny{$f_3$}}
\put(-12,57){\tiny{$f_4$}}

\put(0,0){\vector(3,2){30}}\put(35,19){\tiny{$g_2$}}
\put(30,20){\vector(0,2){20}}\put(35,39){\tiny{$g_3$}}
\put(30,40){\vector(0,2){20}}\put(35,59){\tiny{$g_4$}}

\put(0,20){\vector(3,2){30}} \put(0,40){\vector(3,2){30}}
\put(30,20){\vector(-3,2){30}} \put(30,40){\vector(-3,2){30}}

\put(-30,-20){\tiny{$f_n=\omega_{(n)}^{(1)}$, $g_n=\omega_{(n-1,1)}^{(0)}$}}

\put(-10,-35){\tiny{$\alpha_2=0$, $\beta_2=0$}}
\end{picture}
$
\hskip3truecm
$
\begin{picture}(80,100)
\put(0,0){\circle*{3}} \put(0,20){\circle*{3}}
\put(0,40){\circle*{3}} \put(0,60){\circle*{3}}
\put(0,65){\circle*{1}} \put(0,70){\circle*{1}}
\put(0,75){\circle*{1}}

\put(20,20){\circle*{3}} \put(30,40){\circle*{3}}

\put(0,0){\vector(0,1){20}}\put(-12,-3){\tiny{$f_1$}}
\put(0,20){\vector(0,1){20}}\put(-12,17){\tiny{$f_2$}}
\put(0,40){\vector(0,1){20}}
\put(-12,37){\tiny{$f_3$}}
\put(-12,57){\tiny{$f_4$}}

\put(0,0){\vector(1,1){20}}\put(25,19){\tiny{$g_2$}}
\put(20,20){\vector(1,2){10}}\put(35,39){\tiny{$k_3$}}

\put(0,20){\vector(3,2){30}}

\put(-90,-20){\tiny{$f_n=\omega_{(n)}^{(1)}$, $g_2=\omega_{(1^2)}$, $k_3=\gamma_1\omega_{(2,1)}^{(0)}+\gamma_2\omega_{(2,1)}^{(1)}$}}

\put(-10,-35){\tiny{$\alpha_1+\alpha_2=0$}}

\end{picture}
$
\[
\begin{picture}(80,120)
\put(0,0){\circle*{3}} \put(0,20){\circle*{3}}
\put(0,40){\circle*{3}}

\put(20,20){\circle*{3}} \put(30,40){\circle*{3}}

\put(0,0){\vector(0,1){20}}\put(-12,-3){\tiny{$f_1$}}
\put(0,20){\vector(0,1){20}}\put(-12,17){\tiny{$f_2$}}
\put(-12,37){\tiny{$f_3$}}

\put(0,0){\vector(1,1){20}}\put(25,17){\tiny{$g_2$}}
\put(35,37){\tiny{$k_3$}}

\put(0,20){\vector(3,2){30}}

\put(20,20){\vector(1,2){10}} \put(20,20){\vector(-1,1){20}}

\put(-70,-20){\tiny{$f_n=\omega_{(n)}^{(n-1)}$, $g_2=\omega_{(1^2)}$, $k_3=\gamma_1\omega_{(2,1)}^{(0)}+\gamma_2\omega_{(2,1)}^{(1)}$}}

\put(-10,-35){\tiny{$\alpha_1-\alpha_2 = 0$}}
\end{picture}
\]
\end{example}

\end{document}